\documentclass[reqno,10pt]{amsart}
\usepackage{amssymb}
\usepackage{amsfonts}
\usepackage{amsmath}
\usepackage{bm}
\usepackage{color}
\usepackage{graphicx}
\usepackage[normalem]{ulem}
\usepackage{cancel}
\usepackage{tikz}
\usetikzlibrary{decorations.markings}

\newtheorem{theorem}{Theorem}

\newtheorem{lemma}[theorem]{Lemma}

\theoremstyle{definition}

\newtheorem{remark}[theorem]{Remark}
\numberwithin{equation}{section}
\renewcommand{\d}{\,\mathrm{d}}

\newcommand{\e}{\mathrm{e}}
\newcommand{\epsi}{\varepsilon}

\newcommand{\partialw}{\partial_\mathcal{W}}
\newcommand{\Rz}{{\mathbb R}}
\newcommand{\dt}{\,\mathrm{d}t}
\newcommand{\ds}{\,\mathrm{d}s}

\makeatletter
\@namedef{subjclassname@2020}{\textup{2020} Mathematics Subject Classification}
\makeatother

\newcommand{\bphi}{{\bar{\phi}_\mu}}

\renewcommand{\nu}{}

\usepackage{imakeidx}

\usepackage{version}	

\allowdisplaybreaks

\begin{document}
	
	\title[WIDE approach to doubly nonlinear waves]{Weighted Inertia-Dissipation-Energy approach to
          \\
           doubly nonlinear wave equations}
	
	\author{Goro Akagi}
	\address[Goro Akagi]{Mathematical Institute and Graduate School of Science, Tohoku University, Aoba, Sendai 980-8578, Japan }
	\email{goro.akagi@tohoku.ac.jp} 
	\urladdr{
          http://www.math.tohoku.ac.jp/~akagi/}
        
		\author{Verena B\"ogelein}
	\address[Verena B\"ogelein]{Universit\"at Salzburg,
Fachbereich Mathematik,
Hellbrunner Stra\ss e 34,
5020 Salzburg, Austria}
	\email{verena.boegelein@plus.ac.at} 
\urladdr{https://www.plus.ac.at/mathematics/department/staff/boegelein-verena/?lang=en}
        
	\author{Alice Marveggio}
	\address[Alice Marveggio]{
		Hausdorff Center for Mathematics (HCM), Universit\"at Bonn, Endenicher Allee 60, 53115 Bonn, Germany.}
	\email{alice.marveggio@hcm.uni-bonn.de}
	\urladdr{https://alicemarveggio.github.io/}
	
	\author{Ulisse Stefanelli}
	\address[Ulisse Stefanelli]{University of Vienna, Faculty of Mathematics,
		Oskar-Morgenstern-Platz 1, 1090 Wien, Austria, and Vienna Research Platform on Accelerating Photoreaction Discovery, University of Vienna, W\"ahringerstrasse 17, A-1090 Vienna, Austria, and  Istituto di Matematica
		Applicata e Tecnologie Informatiche
                \textit{{E. Magenes}}, v. Ferrata 1, 27100
		Pavia, Italy.}
	\email{ulisse.stefanelli@univie.ac.at}
	\urladdr{http://www.mat.univie.ac.at/~stefanelli}

	
        \begin{abstract}
       	We discuss a variational approach to  doubly nonlinear wave equations of the form $\rho u_{tt} +  
        g (u_t) - \Delta u +  f (u)=0$. This approach hinges on the
        minimization of a parameter-dependent family of uniformly
        convex functionals over entire trajectories, namely the so-called Weighted Inertia-Dissipation-Energy (WIDE)
	functionals. 
        We prove that the WIDE functionals admit minimizers and that the corresponding Euler-Lagrange system is 
	solvable in the strong sense. Moreover, we check
        that the parameter-dependent minimizers converge, up to subsequences,
        to a solution of the target  doubly nonlinear wave
        equation as the parameter goes to $0$.   
        The analysis relies on
        specific estimates on the WIDE minimizers, on the decomposition of
        the subdifferential of the WIDE functional, and on the
        identification of the nonlinearities in the limit. Eventually,
        we investigate the viscous limit  $\rho \to 0$,  both
        at the functional level and on  that of  the equation.
\end{abstract}
\maketitle


\date{}


\section{Introduction}

The semilinear wave equation $\rho u_{tt}-\Delta u + f  (u)=0$ in the space-time domain $\Omega \times (0,T)$ with $\Omega
\subset \Rz^d$,  $\rho >0$, and $ f =  F'$ monotone can be addressed variationally
by considering minimizers of the global-in-time functionals
$$u \mapsto \int_0^T \int_\Omega \e^{-t/\epsi}\left(\frac{\epsi^2\rho}{2}
  |u_{tt}|^2 + \frac12 |\nabla u|^2 + F(u) \right) \, \d x \,
\d t$$
under given initial and boundary conditions. Indeed, minimizers $u_\epsi$ of the latter
converge, up to subsequences, to solutions of the semilinear wave
equation as $\epsi \to 0$. This is the content of a celebrated conjecture by  {\sc
  De~Giorgi} on the variational resolution of hyperbolic
problems \cite{Degiorgi96},
which has been settled in \cite{St2011} for $T<\infty$ and in
\cite{SeTi2012} in its original formulation with $T=\infty$. The interest of this
approach relies in reformulating the differential problem in terms
of a uniformly convex minimization problem, combined with a limit
passage. This ultimately delivers a variational approximation methodology
for nonlinear wave equations.

Starting from this first positive results, the reach of the De~Giorgi
conjecture has been extended to other classes of
nonlinear hyperbolic problems \cite{Serra-Tilli14}. These include   
nonhomogeneous forcing terms  
\cite{tenta2,tenta3}, general mechanical systems \cite{LiSt}, and
time-dependent domains \cite{DalMaso20}.
The aim of this paper is to consider the
extension of the variational approach to  the case of a nonlinearly
damped wave equation of the form
\begin{equation}\label{eq:dampe}
  \rho u_{tt} +  
  \nu g (u_t) - \Delta u +f (u)=0.
\end{equation}
In addition to the nonlinearity $ f (u)$ on $u$, the latter equation features a second
nonlinear dissipation term $g(u_t)$ with $g$ monotone 
making it a {\it doubly
  nonlinear} wave equation. Correspondingly, the global-in-time
functionals take the form
$$I_{\rho\epsi}: u \mapsto \int_0^T \int_\Omega \e^{-t/\epsi}\left(\frac{\epsi^2\rho}{2}
  |u_{tt}|^2 + \epsi\nu G(u_t)+ \frac12 |\nabla u|^2 + F(u) \right) \, \d x \,
\d t$$
where $G'=g$. 
These functionals feature the weighted sum (via the exponential weight
$t\mapsto \e^{-t/\epsi}$ and powers of the parameter $\epsi$) of an
{\it inertial} term $\rho| u_{tt}|^2/2$, a {\it dissipation} term $\nu
G (u_t)$, and an {\it energy} term $F(u)$. Such global-in-time
functionals are hence usually
referred to as being of \emph{Weighted Inertia-Dissipation-Energy} (WIDE)
type. Correspondingly, the above-mentioned variational approximation strategy
of minimizing the WIDE functionals and then passing to the limit $\epsi \to
0$ is called the {\it WIDE approach}. The relation between the
minimization of $I_{\rho\epsi}$ and the solution to \eqref{eq:dampe}
is revealed by computing the Euler-Lagrange equation for
$I_{\rho\epsi}$. Postponing all necessary details to the coming
sections, we anticipate that this results in the following
fourth-order elliptic-in-time equation
\begin{equation} \label{eq:EL}
\epsi^2 \rho u_{tttt} - 2 \epsi \rho u_{ttt} + \rho u_{tt} -\epsi  \nu
 ( g (u_t))_t + \nu g (u_t) -\Delta u +  f (u) =0.
\end{equation}
In particular, by formally taking the limit $\epsi\to 0$ one recovers the doubly
nonlinear equation \eqref{eq:dampe}. The minimization of
$I_{\rho\epsi}$  hence corresponds to an elliptic-in-time
regularization of \eqref{eq:dampe}. Note that the Euler-Lagrange
equation \eqref{eq:EL} is
{\it not causal}, as its solution $u_\epsi$ at a given time $t$
depends on its values on the interval $(t,T)$ as
well. Causality is restored in the limit $\epsi\to 0$, which is
hence referred to as {\it causal limit} in this context.

The WIDE approach for \eqref{eq:dampe} has already been investigated  
for quadratic  $\psi$. In this case, the resulting limiting problem
\eqref{eq:dampe} is a linearly damped semilinear wave equation. The
amenability of this variational approximation procedure has been
ascertained both in the case $T<\infty$ \cite{LiSt2013}  and for
$T=\infty$ \cite{Serra-Tilli14}. In taking the limit $\epsi \to 0$, 
the linearity of the dissipation
term $ g(u_t)$ is crucially used in
\cite{LiSt2013,Serra-Tilli14}. In particular, the identification
of the nonlinearity $f(u)$ follows by compactness. 

The focus of this paper is in extending the WIDE theory to the
genuinely doubly nonlinear setting by letting $\psi$ be not quadratic.
Indeed, we assume $\psi$ to be convex and of $p$-growth, for
some $2 \leq p < 4$. 
This calls for a number of delicate
extensions of the available arguments. First, the problem will be
abstractly reformulated in Banach spaces, as opposed to the Hilbertian
formulations of \cite{LiSt2013,Serra-Tilli14}. Secondly, in passing to
the limit as $\epsi \to 0$ one needs to identify two
limits. Compactness will still enable to identify the limit in
$f (u)$. For identifying $g(u_t)$ one uses  a lower semicontinuity
argument instead. This is challenging due to the hyperbolic nature
of the problem. At this point, let us refer to \cite{Davoli}, where
the case of a positively $1$-homogeneous $\psi$ (but with $ f =0$) has
been considered in the context of dynamic plasticity, with the help of
tools for rate-independent flows \cite{Mielke-Roubicek}.

Our first main result is the amenability of the WIDE approach in this
doubly nonlinear hyperbolic setting. {\bf Theorem \ref{main_thm}} 
states that, under suitable assumptions, the functionals $I_{\rho\epsi}$ admit
unique minimizers $u_\epsi$, that they are strong solutions of the
Euler-Lagrange equation \eqref{eq:EL},
and that $u_\epsi$ converge to solutions of the doubly nonlinear wave
equation \eqref{eq:dampe} as $\epsi \to 0$, up to subsequences.

We then turn to the investigation of the so-called viscous limit
$\rho\to 0$. This can be alternatively discussed at the level of the functionals
$I_{\rho\epsi}$ or at the level of their minimizers, which we now indicate with
$u_{\rho\epsi}$. {\bf Theorem \ref{viscous_thm}} states that one can take
any limit $(\rho,\epsi) \to (\rho_0,\epsi_0)$ and prove the
convergence of the respective trajectories $u_{\rho\epsi}$ to the
limiting one
$u_{\rho_0\epsi_0}$. In particular, for $(\rho,\epsi) \to (0,0)$ the
minimizers $u_{\rho\epsi}$ converge to the unique solution of the
doubly nonlinear equation $\nu  g (u_t)-\Delta u + f (u) =0$.

Before closing this introduction, let us report on the literature
related to the {\it parabolic} version of the WIDE approach. In
fact,
elliptic-regularization nonvariational techniques for
nonlinear parabolic PDEs are classical
and can be traced
back  to Lions \cite{Lions:63,Lions:63a,Lions:65}, see also 
by Kohn \& Nirenberg \cite{Kohn65},  Olein\u{i}k
\cite{Oleinik}, and 
 the book by {Lions \& Magenes} 
 \cite{Lions-Magenes1}. An early result in a
 nonlinear setting is by Barbu \cite{Barbu75}.

The WIDE approach in the parabolic setting has been pioneered by
Ilmanen \cite{Ilmanen}, for the mean-curvature
flow of varifolds, and Hirano \cite{Hirano94}, for periodic
solutions of gradient flows. Note that  WIDE functionals are mentioned
in the classical textbook by  Evans
\cite[Problem 3, p.~487]{Evans98}.

A variety of different parabolic abstract problems have been tackled by the
WIDE approach, including gradient flows
\cite{akno,BDM15,ms3},
rate-independent flows \cite{MiOr2008,ms2}, doubly-nonlinear flows
\cite{AkMeSt2018,AKSt2010,AkSt2011,AkSt3,Me2}, nonpotential perturbations
\cite{akme,Me} and variational approximations \cite{LiMe}, curves of maximal slope
in metric spaces \cite{RoSaSeSt,RoSaSeSt2,segatti}, and parabolic
SPDEs \cite{ScarStef}. On the more applied side, the WIDE approach has been
applied to microstructure evolution \cite{CO08}, crack propagation
\cite{Larsen-et-al09},
mean curvature flow
\cite{Ilmanen,SpSt}, dynamic plasticity \cite{Davoli}, and the incompressible
Navier-Stokes system \cite{Bathory,OSS}.

This is the plan of the paper. 
We formulate the problem in abstract spaces, collect
assumptions, and formulate our main results, Theorems
\ref{main_thm}-\ref{viscous_thm} in Section
\ref{sec_assumptions}. After collecting some preliminary material in
Section \ref{sec_preliminary}, we prove the existence of a solution to
the Euler-Lagrange problem in Section \ref{sec_existence}. This calls
for an approximation of the WIDE functionals based on the
Moreau-Yosida regularization, the characterization
of their subdifferential, and a limiting procedure with
respect to the approximation parameter. We eventually prove in
Subsection
\ref{sec_minimum} that the WIDE functional admits minimizers. The
passage to the causal limit $\epsi \to 0$ is detailed in Section
\ref{sec_causallimit}. Eventually, the viscous limit $\rho \to 0$ and
its combination with the causal limit $\epsi \to 0$ are discussed in
Section \ref{sec_viscouslimit}.

\section{Assumptions and main results}\label{sec_assumptions}

In this section, we present an abstract formulation for the doubly
nonlinear wave equation \eqref{eq:dampe} and state our main results Theorems
\ref{main_thm} and \ref{viscous_thm}.
Let us start by fixing some  assumptions, which will hold
throughout the paper.

Let $d \in \{2,3\}$ and $\Omega \subset \Rz^d$ be a nonempty, open,
bounded, and Lipschitz domain, and
$V\equiv L^p(\Omega)$ for $2 \leq p < 4$ . 
Moreover, let $H= L^2(\Omega)$ and $ X \equiv H^1_0(\Omega)$, so that $X
\subset V$ densely and compactly. We identify $H=H^\ast$ (dual space), use the symbol $\left\langle \cdot,\cdot
\right\rangle $ for both the duality in $H$ and the duality pairing between $V^{\ast}\equiv L^{p'}(\Omega)$ and $V$, and denote by  $\left\langle \cdot,\cdot
\right\rangle _{X}$ the duality pairing between $X^*$ and $X$.
The symbols $\| \cdot \|$ and $\| \cdot \|_E$ indicate the norms in $H$ and in the generic Banach space $E$, respectively.


We are concerned with the analysis of  WIDE approach 
 to the abstract nonlinear hyperbolic 
Cauchy problem defined as 
\begin{align}
&\rho u''+ \nu \xi(t)+\eta(t)   =0\text{ \ \ \ \ in } V^\ast  \text{ for a.e. }t
 \in (0,T)\text{,}\label{prob1}\\
&\xi(t)   =\mathrm{d}_{V}\psi(u'(t))\text{ \ \ in } V^\ast   \text{ for
	a.e. } t \in (0,T)\text{,}\label{prob2}\\
&\eta(t)   \in\partial\phi(u(t))\text{ }\ \ \ \text{in\ } V^\ast   \text{ for
	a.e. } t \in (0,T) \text{,}\label{prob3}\\
&u(0)   =u_{0}\text{, } \label{prob4}\\
&\rho u'(0)   = \rho u_{1}\text{. } \label{prob5}%
\end{align}
Here, $T > 0$ is some reference final time, the prime
denotes time differentiation, and $\rho$ is a positive
parameter  ($\rho\to 0$ will be considered in Section \ref{sec_viscouslimit}
below).  
The convex functionals $\psi, \phi :V\rightarrow[0,\infty)$ are 
given as 
\begin{align}
  \psi(v)&=
               \left\{
               \begin{array}{ll}
               \displaystyle  \int_\Omega G(v)\, \d x \quad&\text{if} \  G\circ
                                                   v \in
                                                   L^1(\Omega),\\
                 \infty \quad&\text{otherwise},
               \end{array}
                               \right.  \label{def_psi}\\
  \label{def_phi}
\phi (u) &=  \int_{\Omega}  \Big(  \frac12 |\nabla u|^2 +  F(u)\Big)
           \d x\quad \forall u \in V.
\end{align}
We 
denote by $\mathrm{d}_{V}$ the
G\^{a}teaux derivative  and by $\partial $ the subdifferential in
the sense  of convex analysis. 
Finally, $u_0 \in X$ and $u_1  \in X \cap L^{q'}(\Omega)$ with $q'=\frac{2p}{4-p}$ are given initial data. Moreover, we introduce the conjugate exponent $q$ such that $1/q + 1/{q'}=1$.

In the following, we assume $\psi:V\rightarrow\lbrack0,\infty)$ to be twice Gateaux differentiable,  convex, and of $p$-growth. In particular, we assume
that there exists a constant $C_1>0$ such that the following conditions hold
\begin{align}
&\|u\|_{V}^{p}\leq C_1(\psi(u)+1), \ \ \forall u\in V\text{, }\psi(0)=0\text{;}%
\label{coercivity psi}\\
&\|\mathrm{d}_{V}\psi(u)\|_{V^{\ast}}^{p'}\leq C_1(\|u\|_{V}^{p}+1), \ \ \forall u\in V\text{, }p'=p/(p-1);\label{growth condition dpsi} \\
&  \|\operatorname{Hes}(\psi(u))\|^{p''}_{L^{ p''}(\Omega)} \leq C_1(\|u\|_{V}^{p}+1), \ \ \forall u\in V \text{, }p''=p/(p-2), \text{ for }p>2; \label{growth condition hespsi} \\
&  \|\operatorname{Hes}(\psi(u))\|_{L^{ \infty}(\Omega)} \leq C_1, \ \ \forall u\in V \text{, }\text{ for }p=2, \label{growth condition hespsi2}
\end{align}
where $\operatorname{Hes}(\psi)$ stands for the Hessian of $\psi$. 
As a consequence, there exists a constant $C_2>0$ such that
\begin{align}
|u|_{V}^{p}  &  \leq C_2(\left\langle \mathrm{d}_{V}\psi
(u),u\right\rangle +1)\ \ \forall u\in V\text{,}%
\label{coercivity dpsi}\\
\psi(u)  &  \leq\psi(0)+\left\langle \mathrm{d}_{V}\psi(u),u\right\rangle \leq C_2\left(  \|u\|_{V}^{p}+1\right) \ \ \forall u\in
V\text{.} \label{coercivity psi2}%
\end{align}

We assume 
$F\in C^1(\mathbb{R})$ to be convex and $f=  F' \in
C(\mathbb{R})$ to have polynomial growth of order $r-1$, for \( r \in [1,p]\).  
In particular, we ask for some constant $C_3>0$ such that
\begin{align}
\frac{1}{C_3}|v|^r \leq  F(v) + C_3\quad \text{and} \quad |f(v)|^{r'}\leq C_3(1+|v|^r) \quad  \text{for all }v \in \Rz,
\label{eq.growthF}
\end{align}
where $1/r+1/r'=1$. 
The growth assumptions in \eqref{eq.growthF} imply that
$ F$
has at most $r$-growth. We indicate by
$\partial
_{X}\phi_{X}$ the subdifferential from $ X$ to $X^{\ast}$ of the restriction
$\phi_{X}$ of $\phi$ to $X$. 
Note that 
$\partial_X \phi_X$ is single-valued. More precisely, we have $\eta = \partial_X \phi_X  (u) = - \Delta u + f(u)$ in $X^\ast$.
Furthermore, we can deduce the existence of a constant $C_4>0$ such that the following conditions hold: 
\begin{align}
&  D(\phi)\subset X\,;\  |u|_X^2 \leq C_4 (\phi(u)+1)  \ \ \forall u \in D(\phi)\text{,} 
\label{coercivity phi}\\
& \|\eta\|_{X^{\ast}}\leq C_4 (\|u\|_X+\|u\|^{r-1}_{L^r(\Omega)} + 1)\ \ \forall   \eta
 = \partial_{X}\phi_{X}(u) 
. \label{growth condition dphi0} 
\end{align} 



 The WIDE  
approach to the Cauchy problem \eqref{prob1}-\eqref{prob5} 
consists in  defining the parameter-dependent family of WIDE  functionals $I_{\rho\varepsilon
}:L^{p}( 0,T;V)\rightarrow
(-\infty,\infty]$ over entire trajectories as
\begin{align}\label{WIDEfunc}
I_{\rho\varepsilon}(u)  &  =\left\{
\begin{array}
[c]{cl}%
\displaystyle{\int_{0}^{T}\mathrm{e}^{-t/\varepsilon}\Big(\frac{\varepsilon^2	\rho}{2} \int_{\Omega}|u''(t)|^2 \d x  +\varepsilon \nu
	\psi(u'(t))+\phi(u(t))\Big)\mathrm{d}t} & \text{if }u\in K(u_{0}, u_1%
)\text{,}\\
\infty & \text{else,}%
\end{array}
\right.
\end{align}
where we let 
\begin{align*}
K(u_{0}, u_1)  
=\{&u\in H^{2}( 0, T;L^2(\Omega))   \cap W^{1,p}( 0, T;  L^p(\Omega) )
     \cap L^{2}( 0, T;H^1_0(\Omega))   :  \\&u(0)=u_{0}, \ \rho u'(0)= \rho u_1 
\}.
\end{align*}
%
The Euler-Lagrange equation for $I_{\rho\epsi}$ under the constraints $u_{\varepsilon}(0)  =u_{0}$ and $ \rho u'_{\varepsilon}(0)  =\rho u_{1}$ reads 
\begin{align}
&\rho \varepsilon^2 u_\varepsilon''''  - 2 \rho \varepsilon
  u_\varepsilon'''  + \rho u_\varepsilon''  -\varepsilon\nu
  \xi_{\varepsilon}' + \nu\xi_{\varepsilon} +\eta
_{\varepsilon}     =0  \text{ \ \ in } X^{\ast}  , \text{
                          a.e. in }%
  (0,T)\text{,}\label{euler1}\\
&\xi_{\varepsilon}     =\mathrm{d}_{V}\psi(u_{\varepsilon}^{\prime
} )  \text{ \ \  in }  V^{\ast}  ,\text{ a.e. in }  (0,T)\text{,}\label{euler2}\\
& \eta_{\varepsilon}     = \partial_{X}\phi_{ X}(u_{\varepsilon} )
                           \text{ \ \ in } X^{\ast}  ,\text{
                          a.e. in } (0,T)\text{,} \label{euler3}\\
&u_{\varepsilon}(0)  =u_{0}  \text{,} \label{euler4}\\
&\rho u'_{\varepsilon}(0)  = \rho u_{1}  \text{,} \label{euler5} \\ 
&\rho u''_{\varepsilon}(T)  =0  \text{,} \\
& \varepsilon  \rho u'''_{\varepsilon}(T)  -\nu \xi_{\varepsilon}(T) =0  \text{.} \label{euler8} 
\end{align}
In particular, the minimizers of the
WIDE functionals $u_\epsi$  solve an elliptic regularization of the target problem
(\ref{prob1})-(\ref{prob5}).

Our  first  main result reads as follows. 

\begin{theorem}[WIDE variational approach]
	\label{main_thm}  Assume \emph{\eqref{def_psi} 
		-(\ref{growth condition hespsi2}), and \eqref{eq.growthF}}. Then, 
	\begin{itemize}
		\item[i)] The \emph{WIDE} functional $I_{\rho\varepsilon}$  admits a unique global
		minimizer  $u_\epsi \in 
		K(u_{0}, u_1)$. 
		
		\item[ii)]For the unique minimizer
		$u_{\varepsilon}$  of $I_{\rho\varepsilon}$,  
	  by letting $\xi_{\varepsilon} = \mathrm{d}_{V}\psi(u_{\varepsilon}')$ and $\eta_{\varepsilon}= \partial_{X}\phi _{X}(u_{\varepsilon})$,  the triple $(  u_{\varepsilon},\xi_{\varepsilon}, \eta_{\varepsilon})$ belongs to
		\begin{align*}
	 &[W^{4, q}(0,T;X^\ast+ L^{q}(\Omega) ) \cap H^{2}(0,T;L^2(\Omega))\cap W^{1,p}(0,T;V)\cap  L^{2}(0,T;X)] \\
	&\times [W^{1,q}(0,T ;X^{\ast}+L^q(\Omega)) \cap L^{p'}(0,T;V^*)]  \times L^{2}(0,T;X^{\ast} )
	\text{,} \;  
		\end{align*}
		and is a strong solution of the Euler-Lagrange problem
		\emph{(\ref{euler1}%
			)-(\ref{euler8})}. Moreover, the global
                      minimizer of $I_{\rho\epsi}$ 
                      and the strong solution of the Euler-Lagrange
		system coincide.
		\item[iii)] For any sequence $\varepsilon_{k}\rightarrow0$, there exists a subsequence (denoted by the same symbol) such that  
		$(u_{\varepsilon_{k}},\xi_{\varepsilon_{k}}, \eta_{\varepsilon_{k}})\rightarrow (u,\xi, \eta ) $ weakly in 
$$
[W^{2,p'}(0,T;X^*) \cap  W^{1,p}(0,T;V)\cap L^2(0,T;X)] \times
	L^{p'}(0,T;V^*)  \times L^{2}(0,T;X^{\ast} ),
$$  
		where  
		$ (u,\xi , \eta) $ 
		is a strong solution to the doubly nonlinear hyperbolic problem \emph{(\ref{prob1})-(\ref{prob5})}. 
	\end{itemize}
	
\end{theorem}

Theorem \ref{main_thm}.i-ii is proved in Section \ref{sec_existence} by
means of  a regularization procedure whereas the causal $\epsi \to 0$ limit
in Theorem \ref{main_thm}.iii is obtained
 in Section \ref{sec_causallimit}.

Furthermore, we study the viscous limit as $\rho \rightarrow 0$ of
the doubly nonlinear hyperbolic problem
\eqref{prob1}-\eqref{prob5}, recovering the doubly nonlinear
parabolic problem studied in \cite{AkSt2011}, namely
\begin{align}
	&\nu \xi(t)+\eta(t)     =0\text{ \ \ \ \ in } V^\ast  \text{ for a.e. }t
	\in (0,T)\text{,}\label{probvisc1}\\
&	\xi(t)     =\mathrm{d}_{V}\psi(u'(t))\text{ \ \ in } V^\ast   \text{ for
		a.e. } t \in (0,T)\text{,}\label{probvisc2}\\
	&\eta(t)     \in\partial\phi(u(t))\text{ }\ \ \ \text{in\ } V^\ast   \text{ for
		a.e. } t \in (0,T) \text{,}\label{probvisc3}\\
	&u(0)     =u_{0}\text{.} \label{probvisc4}
\end{align}
In this regard, our results are illustrated in the following:
\begin{theorem}[Viscous limit] 
	\label{viscous_thm}  Assume \emph{\eqref{def_psi}
		-(\ref{growth condition hespsi2}), and \eqref{eq.growthF}}. Then,
		\begin{itemize} 
		\item[i)]  The WIDE functionals $I_{\rho\varepsilon}$ $\Gamma$- converge as $\rho \rightarrow 0 $ to
		\begin{align*}
			\bar{I}_{\varepsilon}(u)  &  =\left\{
			\begin{array}
				[c]{cl}%
				\displaystyle{\int_{0}^{T}\mathrm{e}^{-t/\varepsilon}\Big(\varepsilon
				\nu	\psi(u'(t))+\phi(u(t))\Big)\mathrm{d}t} & \text{if }u\in \bar{K}(u_{0}
				)\text{,}\\
				\infty & \text{else,}%
			\end{array}
			\right.
		\end{align*}
		where
		\begin{align*}
			\bar{K}(u_{0})  
			&  =\{W^{1,p}( 0, T;  V ) \cap L^{2}( 0, T;X)   :  u(0)=u_{0}
			\} 
		\end{align*}
	with respect to the strong topology of $L^{p}(0,T;V)  $.
		\item[ii)] Let $(u_{\rho}, \xi_{\rho}, \eta_{\rho})$
                  be a strong solution of the doubly nonlinear hyperbolic
		problem \emph{(\ref{prob1})-(\ref{prob5})} in
		$$
		[W^{1,p}(0,T;V)\cap L^2(0,T;X)] \times
		L^{p'}(0,T;V^*)  \times L^{2}(0,T;X^{\ast} ) .
		$$  
		For any sequence $\rho_{k}\rightarrow0$, there exists a subsequence (denoted by the same symbol) such that  
	$(u_{\rho_k}, \xi_{\rho_k}, \eta_{\rho_k}) \rightarrow (\bar{u},\bar{\xi}, \bar{\eta})$ weakly$^\ast$ in 
		$$
		[W^{1,p}(0,T;V)\cap L^\infty(0,T;X)] \times
		L^{p'}(0,T;V^*)  \times L^{2}(0,T;X^{\ast} ),
		$$  
		where  
		$ (\bar{u},\bar{\xi}, \bar{\eta})$ 	is a strong
                solution to the doubly nonlinear parabolic problem \eqref{probvisc1}-\eqref{probvisc4}.
			\item[iii)] Let $(u_{\varepsilon \rho }, \xi_{\varepsilon\rho}, \eta_{\varepsilon \rho })$  be a strong solution of the Euler-Lagrange problem \emph{(\ref{euler1}%
				)-(\ref{euler8})} belonging to the regularity class of \emph{Theorem \ref{main_thm}.ii}.
		For any pair of sequences $(\varepsilon_{k},
                \rho_{k})\rightarrow (0,0)$, there exists a not
                relabeled subsequence such that  
		$(u_{ \varepsilon _k \rho_k}, \xi_{\varepsilon _k \rho_k},  \eta_{\varepsilon _k \rho_k}) \rightarrow (\tilde{u},\tilde{\xi} , \tilde{\eta})$ weakly in 
		$$
		[W^{1,p}(0,T;V)\cap L^2(0,T;X)] \times
		L^{p'}(0,T;V^*)  \times L^{2}(0,T;X^{\ast} ),
		$$  
		where  
		$ (\tilde{u},\tilde{\xi}, \tilde{\eta})$ is a strong solution to the doubly nonlinear parabolic problem \eqref{probvisc1}-\eqref{probvisc4}.  
	\end{itemize}
      \end{theorem}
      
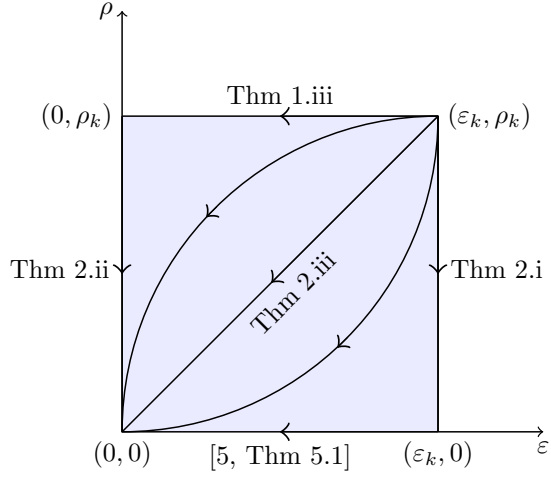
\begin{figure}
	\begin{tikzpicture}[scale=1.4]
		\fill[fill=blue!8!] (0, 0) -- (0, 3) -- (3, 3) -- (3, 0) -- cycle;
		\draw[scale=4,domain=0:2,smooth, ->, line width=0.2mm] (0,0)--(0,1);;
		\node[scale=1,left] at (0,4)   {$\rho$};;
		\node[scale=1,below] at (0,0)   {$(0,0)$};;
		\node[scale=1,below] at (3,0)   {$(\varepsilon_k,0)$};;
		\node[scale=1,right] at (3,3)   {$(\varepsilon_k, \rho_k)$};;
		\node[scale=1,left] at (0,3)   {$(0, \rho_k)$};;
		\draw[scale=4,domain=0:2,smooth, ->, line width=0.2mm] (0,0)--(1,0);;
		\node[scale=1,below] at (4,0)   {$\varepsilon$};;
		\draw[scale=3,domain=0:2,smooth, -, line width=0.2mm] (1,1)--(0,1);;
		\draw[scale=3,domain=0:2,smooth, -, line width=0.2mm] (1,1)--(1,0);;
		\draw[scale=3,domain=0:2,smooth, -, line width=0.2mm] (1,1)--(0,0);;
		\node[scale=1,above] at (1.5,3.03)   {Thm \ref{main_thm}.iii};;
		\node[scale=1,right] at (3.03,1.55)   {Thm \ref{viscous_thm}.i};;
		\node[scale=1,left] at (-0.03,1.55)   {Thm \ref{viscous_thm}.ii};;
		\node[scale=1,right, rotate=45] at (1.2,0.9)   {Thm \ref{viscous_thm}.iii};;
		\node[scale=1,below] at (1.5,-0.03)   {\cite[Thm 5.1]{AkSt2011}};;
		
		\draw[line width=0.2mm] (3,0) -- (3,3) arc[start angle=90, end angle=180,radius=3cm,  line width=0.23mm] -- (0,0);
			\draw[line width=0.2mm] (0,3) -- (0,0) arc[start angle=270, end angle=360,radius=3cm] -- (3,3);
		
		\draw[
		decoration={markings,mark=at position 1 with {\arrow[scale=2]{>}}},
		postaction={decorate},
		shorten >=0.1pt
		]
		(3,0) -- (1.5,0);
		
		\draw[
		decoration={markings,mark=at position 1 with {\arrow[scale=2]{>}}},
		postaction={decorate},
		shorten >=0.1pt
		]
		(0,3) -- (0,1.5);
		
		\draw[
		decoration={markings,mark=at position 1 with {\arrow[scale=2]{>}}},
		postaction={decorate},
		shorten >=0.1pt
		]
		(3,3) -- (3,1.5);
		
		\draw[
		decoration={markings,mark=at position 1 with {\arrow[scale=2]{>}}},
		postaction={decorate},
		shorten >=0.1pt
		]
		(3,3) -- (1.5,3);
		
		\draw[
		decoration={markings,mark=at position 1 with {\arrow[scale=2]{>}}},
		postaction={decorate},
		shorten >=0.1pt
		]
		(3,3) -- (1.415,1.415);
		
		\draw[
		decoration={markings,mark=at position 1 with {\arrow[scale=2]{>}}},
		postaction={decorate},
		shorten >=0.1pt
		]
		(0.9,2.14) -- (0.8,2.04);
		
		\draw[
		decoration={markings,mark=at position 1 with {\arrow[scale=2]{>}}},
		postaction={decorate},
		shorten >=0.1pt
		]
		(2.145,0.9) -- (2.045,0.8);
		

	\end{tikzpicture}
	\label{figure:diag}
	\caption{  Theorems \ref{main_thm}-\ref{viscous_thm}:
          illustration of the various limits with respect to $\epsi$
          and $\rho$.}
      \end{figure}
      
Theorem \ref{viscous_thm} is proved in Section
\ref{sec_viscouslimit}. More precisely, Theorem \ref{viscous_thm}.i is
obtained in Subsection \ref{subsec:viscgammconv}, whereas the
proof of Theorem \ref{viscous_thm}.ii is given in Subsection
\ref{subsec:viscconv}. Eventually, Theorem \ref{viscous_thm}.iii
can be proved by combining the arguments of Sections
\ref{sec_causallimit}-\ref{sec_viscouslimit}.

\section{Preliminary materials}\label{sec_preliminary}
 In this section, we present a tool  
for further use.

\begin{lemma}[Weighted limsup tool] \label{lemma:doublelimsup}
Let $A: D(A) \subset V\rightarrow 2^{V^*}$ be a maximal monotone operator. Let $v_\varepsilon \rightarrow v$ weakly in $L^p(0,T; V)$ as $\varepsilon \rightarrow 0$. Let $w_\epsilon(t) \in Av_\varepsilon(t)$ for almost every $t \in (0,T)$ be such that $w_\epsilon \rightarrow w$ weakly in $L^{p'}(0,T; V^*)$ as $\varepsilon \rightarrow 0$. Moreover, assume that the following inequality holds
\begin{align}
	\limsup_{\varepsilon \rightarrow 0} \int_0^T (T-t) \langle w_\varepsilon, v_\varepsilon\rangle \dt \leq  \int_0^T (T-t)\langle w, v\rangle \dt . \label{eq:hplimsup}
\end{align}
Then, $w(t) \in Av(t)$ for almost every $t \in (0,T)$.
\end{lemma}

\begin{remark}\label{rem:limsup}
	Note that  
		\begin{align*}
		\int_0^T \int_0^t f(s)\ds \dt = 	\int_0^T (T-t)
                  f(t) \dt \,, \quad \text{ for any } f \in L^1(0,T). 
	\end{align*}
As a consequence, assumption \eqref{eq:hplimsup} of Lemma \ref{lemma:doublelimsup} is equivalent to
	\begin{align}
		\limsup_{\varepsilon \rightarrow 0} \int_0^T \int_0^t \langle w_\varepsilon, v_\varepsilon\rangle_V \ds \dt \leq  \int_0^T \int_0^t \langle w, v\rangle_V \ds \dt \, . \label{eq:hplimsup0}
	\end{align}
\end{remark}

\begin{proof}[Proof of Lemma \ref{lemma:doublelimsup}]
	Fix $t_0 \in (0,T)$. Let $\delta>0$ be such that $t_0>\delta$
        and $T-t_0 >\delta$, and let $B_\delta(t_0):= (t_0 - \delta, t_0+\delta)$.
    For any $\tilde{v} \in V$ and $\tilde{w} \in A\tilde{v}$, we define
    $$\hat{v}(t):=
    \left\{
      \begin{array}{ll}
        v(t)  \quad& \text{ for } t \notin B_\delta(t_0), \\
        \tilde{v} & \text{ for } t \in B_\delta(t_0),
      \end{array}\right.$$
                and
                $$\hat{w}(t):=
                \left\{
                  \begin{array}{ll}
                    z(t) \quad& \text{ for } t \notin B_\delta(t_0), \\
                    \tilde{w} &\text{ for } t \in B_\delta(t_0), 
\end{array}	\right.$$ where $z(t) \in  Av(t)$. 

Using the monotonicity of $A$, \eqref{eq:hplimsup} and the weak
convergence properties of $w_\varepsilon$ and of $v_\varepsilon$,  we
deduce that
\begin{align*}
0 \leq \limsup_{\varepsilon \rightarrow 0} \int_0^T  (T-t)\langle w_\varepsilon - \hat{w}, v_\varepsilon - \hat{v}\rangle  \dt \leq  \int_0^T (T-t) \langle w - \hat{w}, v  - \hat{v}\rangle  \dt .
\end{align*}
From the definition of $\hat{v}$ we deduce
\begin{align*}
	\int_0^T &(T-t) \langle w - \hat{w}, v  - \hat{v}\rangle \dt =  \int_{ B_\delta(t_0)} (T-t) \langle w - \tilde{w}, v  - \tilde{v}\rangle  \dt 	.
\end{align*}
Since the left-hand side is nonnegative, it follows that
\begin{align*}
 0 \leq \frac{1}{2 \delta}\int_{ B_\delta(t_0)}
  (T-t)\langle w - \tilde{w}, v  - \tilde{v}\rangle \d t ,
\end{align*}
where the integral in the right-hand side converges to $(T-t_0) \langle w(t_0)
- \tilde{w}, v(t_0)- \tilde{v} \rangle$ as $\delta \rightarrow 0$
for almost all $t_0 \in (0,T)$, due to the  Lebesgue  Differentiation Theorem.
As a consequence, we obtain \[
 \langle w(t_0) - \tilde{w}, v(t_0)- \tilde{v}\rangle \geq 0 
\]
  for almost all $t_0 \in (0,T)$. Finally, recalling that $\tilde{w}\in A \tilde{v}$, by using the maximal monotonicity of $A$ in $V \times V^\ast$, we can conclude that $w(t) \in Av(t)$ for almost all $t \in (0,T)$. 
\end{proof}

\section{Existence of solutions to the Euler-Lagrange
  problem} \label{sec_existence}

This section focuses on the solution of the Euler-Lagrange
  problem \eqref{euler1}-\eqref{euler8} and on the minimization of the WIDE
  functionals. At first, we introduce some approximation of the
  functionals in Subsection \ref{sec:approx} depending on two approximation parameters $\mu, \lambda >0$, and investigate their
  subdifferential in Subsection \ref{sec:subdifferential}. These
  approximated functionals admit minimizers. The corresponding
  Euler-Lagrange problem is in Subsection \ref{sec:EL}. We derive a
  priori estimates in Subsection \ref{subsec_a priori estimates} which
  allow to pass to the limit in the approximation first as $\mu \rightarrow 0$ (Subsection
  \ref{sec:limitmu}) and then as $\lambda\rightarrow 0$ (Subsection
  \ref{sec:limitlambda}), and to find a solution to  the Euler-Lagrange
  problem \eqref{euler1}-\eqref{euler8}. Eventually, such solutions
  are checked to correspond to minimizers of the WIDE functionals in
  Subsection \ref{sec_minimum}. 
  
\subsection{Approximating functional $I_{\rho\varepsilon}^{\lambda\mu}$} \label{sec:approx}
Recall that $H=L^2(\Omega)$, define the spaces
\begin{align} \label{spaces}
\mathcal{H} = L^{2}( 0,T;H), \quad \mathcal{V} = L^{p}( 0,T;V)\,, \quad \mathcal{W} = H^{2}( 0,T;H) 
\,,
\end{align}
and the approximating functional
$I^{\lambda\mu}_{\rho\varepsilon} : \mathcal{H}  \rightarrow(-\infty,\infty]$
as 
\begin{align}
&I_{\rho\varepsilon }^{\lambda\mu} (u)  =\left\{
\begin{array}
[c]{cl}%
\displaystyle{\int_{0}^{T}\mathrm{e}^{-t/\varepsilon}\Big(\frac{\varepsilon^2	\rho}{2} \|u''(t)\|^2  +\varepsilon    \bar \psi_\lambda (u^\prime(t)) + \bphi (u(t)) \Big)\mathrm{d}t}  & \text{if }u\in K_\lambda(u_{0}, u_1%
)\text{,}\\
\infty & \text{else,}%
\end{array}
\right.\label{WIDEfuncelambda}
\\[3mm] \notag
&K_\lambda(u_{0}, u_1)   = \{u\in    \mathcal{W}  : u(0)=u_{0}, \rho u'(0)= \rho u_1 \} .
\end{align}
Here, $\bar \psi_\lambda$ denotes the Moreau-Yosida
regularization at the level $\lambda>0$ of the extension $\bar \psi: H \rightarrow [0,\infty]$, given by
\[\bar \psi(v) = \begin{cases}
	\psi(v) \quad \text{ if } v \in V, \\
	\infty \quad \text{ else},
\end{cases}\]
namely, 
\begin{align}
	\bar \psi_{\lambda}(v):=\inf_{w\in {H}}\left( \frac{1}{2\lambda} \|v-w\|^{2} 
	+ \bar \psi(w) \right)  
	=\frac{1}{2\lambda} \|v- \bar J_{\lambda}v\|^{2} + \psi(\bar J_\lambda v) \text{,} \label{psi_lambda} 
\end{align}
where $\bar J_{\lambda}$ is the resolvent of  $\partial_{{H}}\bar \psi_{ \lambda}$ at level $\lambda$, namely the solution operator $\bar J_\lambda : v \mapsto \bar J_\lambda v$ to 
\begin{align}
	( v-  \bar J_\lambda v  )\in \lambda  \partial_{H} \psi
	(\bar J_\lambda v)   
	\text{ \ \ in } H  ,
	\label{eq:barresolvent} 
\end{align}
for any $v \in {H}$.
In particular, one has that $\bar \psi_\lambda \in C^{1,1}(H), \,\partial_{H} \bar \psi_\lambda \in C^{0,1}(H;H)$, and $\operatorname{Hes}(\bar \psi_{ \lambda}\lambda ) \in L^\infty(H; \mathcal{L}(H;H))$ (Hessian). 
Moreover, $\bphi$ denotes the Moreau-Yosida
regularization at level $\mu>0$ of the extension $\bar \phi : H \rightarrow [0,\infty]$ of $\phi$,  namely,
\begin{align}
\bphi(u):=\inf_{v\in {H}}\left( \frac{1}{2\mu}\|u-v\|^{2} 
+ \phi(v) \right)  
=\frac{1}{2\mu}\|u-J_{\mu}u\|^{2} + \phi(J_{ \mu} u) \text{,} \label{phi_lambda} 
\end{align} 
where $J_{\mu}$ is the  resolvent of  $\partial_{{H}}\bphi$ at level $\mu$, namely, the solution operator $J_{\mu} : u \mapsto J_{\mu} u$ to 
\begin{align}
 ( u-  J_{\mu} u  )  \in { \mu}  \partial_{H} \phi
  (J_{\mu} u)   
  \text{ \ \ in } H ,
\label{eq:resolvent} 
\end{align}
for any $u \in {H}$. 
 Recall that $\partial_H \bar \psi_\lambda (v)= (v - \bar J_\lambda
v)/\lambda$  and $\partial_H \bphi (v)= (v - J_{\mu} v)/{\mu}$. As 
$
D( I^{\lambda\mu}_{\rho\varepsilon})= K_\lambda(u_0, u_1),
$
$I^{\lambda\mu}_{\rho\varepsilon}$ can be decomposed as
\begin{align}
I_{\rho \varepsilon}^{\lambda\mu}=\bar{I}^\lambda_{\rho \varepsilon}+\Phi_{\varepsilon\mu}%
\text{,} \label{Idecomposition}
\end{align}
where the functionals $\bar{I}^\lambda_{\rho\varepsilon}, \Phi_{\varepsilon
 \mu}: \mathcal{H} \rightarrow [0,\infty]$ are defined by
\begin{align}
\label{Ibar}
& \bar{I}^\lambda_{\rho \varepsilon}(u)=\left\{
\begin{array}
[c]{cl}%
\displaystyle{\int_{0}^{ T }\mathrm{e}^{-t/\varepsilon}\Big(\frac{\varepsilon^2	\rho}{2} \|u''(t)\|^2  +\varepsilon \nu  \bar \psi_\lambda(u^\prime(t)) \Big)\mathrm{d}t}   & \text{if } u\in {K}_\lambda(u_{0}, u_1%
)\text{,}\\
\infty & \text{else,}%
\end{array}
\right. 
\end{align}
and 
\begin{align}
\label{Phi}
\Phi_{\varepsilon\mu}(u)& =
\displaystyle{\int_{0}^{T}\mathrm{e}^{-t/\varepsilon} \bphi (u) \, \mathrm{d}t}   
\end{align}
with domains 
\begin{align*}
D(\bar{I}^\lambda_{\rho \varepsilon})    ={K}_\lambda(u_{0}, u_1)  ,\quad 
D(\Phi_{\varepsilon \mu})   =  \mathcal{H}. 
\end{align*}

The functional $I_{\rho\varepsilon}^{\lambda\mu}$ is proper, 
lower semicontinuous, and convex in $\mathcal H$. 
Moreover, thanks to the Poincar\'e inequality, it is coercive on $\mathcal{H}$.  The
Direct Method ensures that $I_{\rho\varepsilon}^{\lambda\mu}$ admits
a minimizer $u_{\varepsilon\lambda \mu} \in  K_\lambda(u_{0}, u_1)$. 

 \subsection{Representation of
   subdifferentials} \label{sec:subdifferential}
 In order to derive the Euler-Lagrange equation for
$I^{\lambda\mu}_{\rho\varepsilon}$, we prepare here some representation results. Recalling \eqref{spaces},
we denote by $\partial_\mathcal{H}$ and $\partial_\mathcal{W}$ the subdifferentials in
the sense of convex analysis from $\mathcal{H}$ to
$\mathcal{H}^{\ast}$ and from $\mathcal{W}$ to $\mathcal{W}^{\ast}$, respectively. 

First, note that we can further decompose the functional $ \bar{I}_{\rho\varepsilon}^\lambda: \mathcal{H}   \rightarrow [0, \infty]$ as
\[
 \bar{I}_{\rho\varepsilon}^\lambda =  {Y}_{\rho\varepsilon}^\lambda + {Y}_{(0)} \,,
\]
where 
\begin{align}
\nonumber 
& {Y}_{\rho\varepsilon}^\lambda(u)=\left\{
\begin{array}
[c]{cl}%
\displaystyle{\int_{0}^{T}\mathrm{e}^{-t/\varepsilon} \left( \frac{\varepsilon^2	\rho}{2} \|u''(t)\|^2 + \varepsilon \nu \bar\psi_\lambda(u^\prime(t)) \right) \, \mathrm{d}t}   & \text{if } u\in \mathcal{W} \text{,}\\
\infty & \text{else,}%
\end{array}
\right. 
\\[3mm] \nonumber 
&{Y}_{(0)}(u) = \left\{
\begin{array}
[c]{cl}%
0  & \text{if } u \in \mathcal{W} \,, \, u(0)=u_{0}\,, \text{ and } \rho u'(0)= \rho u_1
\text{,}\\
\infty & \text{else.}%
\end{array}
\right. 
\end{align}
The functional ${Y}^\lambda_{\rho\varepsilon} $ is Gateaux differentiable in $\mathcal{W}$
, in particular we have
\begin{align*}
	\langle \mathrm{d}_\mathcal{W} {Y}^\lambda_{\rho\varepsilon} (u),
  e \rangle_\mathcal{W} &=  \int_0^T \left( \mathrm{e}^{-t/\varepsilon}
 \varepsilon^2 \rho  \langle u''(t) , e''(t)
                          \rangle  +  \mathrm{e}^{-t/\varepsilon}
                          \varepsilon \langle \partial_H \bar \psi_\lambda(u'(t)) ,
                          e'(t)  \rangle \right)\, \mathrm{d} t ,
                          \; \; \; \forall e \in \mathcal{W} ,
\end{align*}
where 	$\langle \cdot, \cdot \rangle_\mathcal{W} $ stands for the duality pairing between $\mathcal{W}^\ast$ and $\mathcal{W}$.
On the other hand, we have
\begin{align*}
\langle f, e\rangle_{\mathcal{W}}=0 \quad \text { for all }[u, f] \in \partial_\mathcal{W} {Y}_{(0)}\ \text { and } e \in \mathcal{W} \text { with } e(0)=e'(0)=0.
\end{align*}
Since $D({Y}^\lambda_{\rho\varepsilon} ) =  \mathcal{W}$, we deduce that
\[
 \partialw \bar{I}_{\rho\varepsilon}^\lambda = \mathrm{d}_\mathcal{W} {Y}^\lambda_{\rho\varepsilon}  + \partialw {Y}_{(0)}\,,
\]
with domain
$
D(\partialw \bar{I  }^\lambda_{\rho\varepsilon})  
= K_\lambda(u_0,u_1) \,.
$


Let $u \in D(\partial_{\mathcal H} \bar{I}^\lambda_{\rho \varepsilon})$. As we have $D(\bar{I}^\lambda_{\rho\varepsilon})  \subset \mathcal{W}  \subset
\mathcal{H} $, we conclude that $ \partial_\mathcal{H}\bar{I}^\lambda_{\rho \varepsilon}
\subset \partial_{\mathcal{W}} \bar{I}^\lambda_{\rho\varepsilon} $. 
Letting $f \in  \partial_{\mathcal{H}}  \bar{I}^\lambda_{\rho\varepsilon}(u) $, since $f  \in  \mathrm{d}_\mathcal{W}  {Y}^\lambda_{\rho\varepsilon} (u)  + \partialw {Y}_{(0)}(u) $, we obtain
\begin{align}\label{eq:raprsubdiff}
	\int_{0}^{T} \e^{-t / \varepsilon} \varepsilon^2 \rho \left\langle u^{\prime \prime}(t), e^{\prime \prime}(t)\right\rangle + \e^{-t / \varepsilon} \varepsilon \nu \left\langle \partial \bar\psi_\lambda\left(u^{\prime}(t)\right), e^{\prime}(t)\right\rangle \, \mathrm{d} t &=\int_{0}^{T}\langle f(t), e(t)\rangle\, \mathrm{d} t, \\
\forall e \in  \mathcal{W} \text{ such that } e(0)&=e'(0)=0. \notag
\end{align}
By setting
\[
g(t):= - \int_t^T f (s) \ds , 
\]
we have $g \in W^{1,2}(0,T;H)$ and, for all $\varphi \in C^\infty_c(0,T)$,
\begin{align*}
	\int_0^T  \varphi(t)  g(t) \dt 
	&= - \int_0^T \Big(\int_0^s \varphi(t) \, \dt \Big) f(s) \, \ds \\
	&= -\int_{0}^{T} \left(\e^{-t / \varepsilon} \varepsilon^2 \rho u^{\prime \prime}(t) \varphi^{\prime}(t) +  \e^{-t / \varepsilon} \varepsilon \nu \partial_H \bar \psi_\lambda\left(u^{\prime}(t)\right)\varphi(t) \right)\, \mathrm{d} t ,
\end{align*}
where in the last equality we used \eqref{eq:raprsubdiff}. 
As a next step, we observe that $\partial_H \bar \psi_\lambda (u') \in W^{1,2}(0,T;H)$, due to $(\partial_H \bar \psi_\lambda (u'(t)) )' = \operatorname{Hes}(\bar \psi_\lambda (u'(t)) )u''(t)$ and the fact that $\operatorname{Hes}(\bar \psi_\lambda (u'(t)) )$ is bounded.
Hence, $\left(g +  \e^{- t / \varepsilon} \varepsilon \nu \partial_H \bar \psi_\lambda  (u^{\prime}) \right)\in  W^{1,2}(0,T;H) $. 
As a consequence, we obtain
\begin{align*}
	\int_0^T  \left(g(t) + \e^{-t / \varepsilon} \varepsilon \nu \partial_H \bar \psi_\lambda \left(u^{\prime}(t)\right) \right) \varphi(t)  \dt
	&= - \int_{0}^{T} \e^{-t / \varepsilon} \varepsilon^2 \rho u^{\prime \prime}(t) \varphi^{\prime}(t)\, \mathrm{d} t,
\end{align*}
whence $\e^{-t / \varepsilon} \varepsilon^2 \rho u^{\prime \prime} \in W^{1,2}(0,T;H)$ with $\frac{\mathrm{d}}{\mathrm{d} t } \left(\e^{-t / \varepsilon} \varepsilon^2 \rho u^{\prime \prime} \right) = g + \e^{-t / \varepsilon} \varepsilon \nu \partial_H \bar \psi_\lambda \left(u^{\prime}\right)$. 
In particular,  $(\e^{-t / \varepsilon} \varepsilon^2 \rho u^{\prime \prime \prime} $ $ - \e^{-t / \varepsilon} \varepsilon\rho u^{ \prime \prime} )  \in \mathcal{H}$ and $\e^{-t / \varepsilon} \varepsilon^2 \rho u^{\prime \prime \prime}  \in \mathcal{H}$ as $u^{\prime \prime} \in \mathcal{H}$. 
It follows from \eqref{eq:raprsubdiff} again that 
\begin{align*}
\lefteqn{
- \int_{0}^{T} \e^{-t / \varepsilon} \varepsilon^2 \rho u^{\prime \prime \prime}(t) \varphi^{\prime }(t) \, \d t
}\\
&= \int^T_0 f(t) \varphi(t) \, \d t
- \int^T_0 \e^{-t/\varepsilon} \varepsilon \partial_H \bar\psi_\lambda\left(u^{\prime}(t)\right)  \varphi'(t)  \, \d t- \int^T_0 \e^{-t/\varepsilon} \varepsilon \rho u^{\prime \prime}(t) \varphi^{\prime}(t) \, \d t,\\
&\qquad \qquad \qquad \qquad \qquad \qquad \qquad \mbox{ for all } \ \varphi \in C^\infty_0(0,T), 
\end{align*}
which yields $\e^{-t/\varepsilon} \varepsilon^2 \rho u^{\prime \prime \prime} - \e^{-t/\varepsilon} \varepsilon \partial_H \bar\psi_\lambda(u^{\prime}) - \e^{-t/\varepsilon} \varepsilon \rho u^{\prime \prime} \in W^{1,2}(0,T;H)$, and hence, due to the facts obtained so far, we infer that $u \in W^{4,2}(0,T; H)$.

Therefore, integrating by parts in \eqref{eq:raprsubdiff},
 we have
$$
f(t) =\frac{\mathrm{d}^2}{\mathrm{d} t^2}\left(\e^{-t / \varepsilon} \varepsilon^2 \rho u''(t)\right)-\frac{\mathrm{d}}{\mathrm{d} t}\left(\e^{-t / \varepsilon} \varepsilon \nu \partial_H  \bar\psi_\lambda\left(u^{\prime}(t)\right)\right) \in  H
$$
for almost all $t \in (0,T)$ and also $\varepsilon^2\rho u''(T)= \varepsilon^2\rho u'''(T) - \varepsilon \partial_H  \bar\psi_\lambda(u^{\prime}(T))=0$, due to the arbitrariness of $e(T ), e'(T) \in H$, as well as $\rho u''(T)=\varepsilon \rho u'''(T) - \partial_H  \bar\psi_\lambda(u^{\prime}(T))=0$.

Finally, we can conclude that
\begin{align} \label{inclusion domain subdiff}
D(\partial_{\mathcal{H}} \bar{I}^\lambda_{\rho \varepsilon})  \subset   \{  u \in \mathcal{W} :  \, 
 u \in W^{4,2}(0,T; H),     u(0)=u_0, \rho u'(0)= \rho u_1, &
                                                         \notag\\  \rho u''(T)=0,
                                                          \varepsilon \rho u'''(T) -  \partial_H 
                                                          \bar \psi_\lambda(u^{\prime}(T))=0 &\} ,
\end{align}
and
\begin{align}
	\partial_{\mathcal{H}} \bar{I}^\lambda_{\rho \varepsilon}(u)(t) = \frac{\mathrm{d}^2}{\mathrm{d}t^2}\left(\e^{-t / \varepsilon} \varepsilon^2 \rho u''(t)\right) -\frac{\mathrm{d}}{\mathrm{d} t}\left(\e^{-t / \varepsilon} \varepsilon \nu \partial_H \bar \psi_\lambda\left(u^{\prime}(t)\right)\right) \quad \text{for a.e. } t \in (0,T)	.
\end{align}
We observe that the inclusion $\subset $ in \eqref{inclusion domain subdiff} can be replaced by an equality, as the reverse inclusion $\supset$ is straightforward.


Note that  $\partial_{\mathcal{H}}\Phi_{\varepsilon \mu}: \mathcal{H}  \rightarrow\mathbb{R}$ is
demicontinuous (i.e., strong-weak continuous) and single-valued.  
As a consequence, we have that $D(\Phi_{\varepsilon \mu})= \mathcal{H} $.  As  $\partial_{\mathcal{H}} \bar{I}_{\rho \varepsilon}  + \partial_{\mathcal{H}}\Phi_{\varepsilon \mu}$ is
maximal monotone in $\mathcal{H\times H}$,  we obtain
\begin{equation}
\partial_{\mathcal{H}} I_{\varepsilon\rho}^{\lambda\mu} =\partial
_{\mathcal{H}}\bar{I}^\lambda_{\rho \varepsilon}  +\partial_{\mathcal{H}%
}\Phi_{\varepsilon \mu} \text{.} \label{subdiff}%
\end{equation} 
Note that the minimizer of $ I^{\lambda\mu}_{\varepsilon\rho}$ is unique. It
hence coincides with the strong solution of \eqref{euler:lambda} below. 

\subsection{Euler-Lagrange equation for \(I_{\rho\varepsilon}^{\lambda \mu}\)} \label{sec:EL}

 Thanks to the decomposition \eqref{subdiff}, the minimizer  
 $u_{\varepsilon \lambda \mu}$ of \(I_{\rho\varepsilon}^{\lambda \mu}\) 
 fulfills 
 \[
0\in\partial
 _{\mathcal{H}}\bar{I}^\lambda_{\rho\varepsilon}(u_{\varepsilon \lambda \mu})  +\partial_{\mathcal{H}%
 }\Phi_{\varepsilon \mu} (u_{\varepsilon \lambda \mu}) \text{.} 
 \]
In particular, the following holds 
 \begin{align}
&\rho \varepsilon^2 u''''_{\varepsilon \lambda \mu} - 2 \rho
                 \varepsilon u'''_{\varepsilon \lambda \mu}  + \rho
                 u''_{\varepsilon \lambda \mu}   + \eta_{\varepsilon
                 \lambda \mu}  -\varepsilon  \nu \xi_{\varepsilon
                 \lambda \mu}' +  \nu \xi_{\varepsilon \lambda \mu}  =0
                 \; \text{  in } H, \,\text{ a.e. in } 
                 (0,T) \text{,}\label{euler:lambda}
\\
&  \eta_{\varepsilon \lambda \mu}   =  \partial_H \bphi
     (u_{\varepsilon \lambda \mu}  )    = - \Delta J_{\mu}
     u_{\varepsilon \lambda \mu}  + f(J_{\mu} u_{\varepsilon \lambda \mu}
     )   \text{ \ in } H,\text{ a.e. in } (0,T)\text{,} \label{euler:etalambda}\\
 &  \xi_{\varepsilon \lambda \mu} =\partial_H \bar \psi_\lambda ( u_{\varepsilon\lambda \mu}' )
   \text{ in } H, \text{ a.e. in }   (0,T), \label{euler:xilambda}\\
 &u_{\varepsilon\lambda \mu}(0)  =u_{0} \text{,} \label{euler:0} \\
& \rho u'_{\varepsilon\lambda \mu}(0)  = \rho u_{1} \text{,} \label{euler:1} \\
&\rho u''_{\varepsilon\lambda \mu}(T)  =0\text{,} \label{euler:3}  \\
& \varepsilon  \rho u'''_{\varepsilon\lambda\mu}(T)  - \nu \xi_{\varepsilon\lambda\mu}(T)  
                                                                    =0 \label{euler:4} \text{.}
 \end{align}

 \subsection{A priori estimates} \label{subsec_a priori estimates} 
 We now derive a priori estimates for $u_{\varepsilon \lambda \mu}$ which
 will eventually allow us to pass to the limit as $\mu \rightarrow0$ in Subsection \ref{sec:limitmu} and then $\lambda\rightarrow0$ in Subsection \ref{sec:limitlambda}, 
 then as $\varepsilon\rightarrow0$  in Section
 \ref{sec_causallimit}, and finally as $\rho \rightarrow 0$
 in Section \ref{sec_viscouslimit}.   In what follows, the symbol $C$ will
 denote a generic positive constant independent of $\mu$, $\lambda$, $\varepsilon$, and $\rho$, possibly varying from line to line. 
 Occasionally, specific dependencies of the constant $C$ will be indicated.
 Furthermore, for the sake of brevity, we write $(u, \xi, \eta)$ instead of $(u_{\varepsilon \lambda \mu},\xi_{\varepsilon \lambda \mu}, \eta_{\varepsilon \lambda \mu}) $ in the following.

%
Testing \eqref{euler:lambda} with $u' - u_1$, integrating over $\Omega \times (0,t)$ for an arbitrary $t \in (0,T]$ and then integrating by parts, we obtain
 \begin{align} \label{eq:nestedt}
 &  \frac{\rho \varepsilon^2}{2} \|u''(0)\|^2 -
   \frac{\rho \varepsilon^2}{2} \|u''(t)\|^2 
 + \rho \varepsilon^2 \langle u'''(t), u'(t) - u_1 \rangle  
   \nonumber\\
   &- 2 \rho \varepsilon \langle u''(t), u'(t) - u_1 \rangle
 + 2 \rho  \varepsilon \int_0^t \|u''(s)\|^2\ds \notag \\
 &+\frac{\rho }{2} \|u'(t)-u_1\|^2
 - \varepsilon \nu \langle \xi(t), u'(t)- u_1 \rangle 
 + \int_{0}^{t} \nu \langle \xi(s), u'(s)- u_1 \rangle \ds
  + \varepsilon \nu \bar \psi_\lambda(u'(t))+\bphi (u(t))\notag \\
  &= \varepsilon\nu \bar \psi_\lambda (u_1) + \bphi (u_0) + \int_{0}^t \langle \eta(s), u_1 \rangle \ds .
 \end{align}
Observe that for $t=T$, the estimate \eqref{eq:nestedt} reduces to 
\begin{align} \label{eq:nestedT}
&\int_{0}^{T} \nu \langle \xi(t), u'(t)- u_1 \rangle \dt + \frac{\rho \varepsilon^2}{2} \|u''(0)\|^2 + 2 \rho \varepsilon \int_0^T \|u''(t)\|^2 \dt +\frac{\rho }{2} \|u'(T)-u_1\|^2\notag \\ &+ \varepsilon \nu \bar \psi_\lambda(u'(T))+\bphi(u(T))= \varepsilon \nu \bar \psi_\lambda(u_1) + \bphi(u_0) + \int_{0}^T \langle \eta(t), u_1 \rangle \dt .
\end{align}
Integrating \eqref{eq:nestedt} over $(0,T)$, then integrating by parts and summing it to \eqref{eq:nestedT}, we obtain
 \begin{align} \label{eq:nested}
&  \frac{\rho \varepsilon^2(T+1)}{2} \|u''(0)\|^2
+  \frac{\rho (4- 3\varepsilon) \varepsilon}2  \int_{0}^T
                                   \|u''(t)\|^2 \dt  + 2 \rho  \varepsilon \int_{0}^T \int_0^t \|u''(s)\|^2 \ds \dt \notag \\
&+ \frac{\rho ( 1-2\varepsilon) }2 {\|u'(T)-u_1\|^2}
+ \int_{0}^T \frac{\rho }2 {\|u'(t)-u_1\|^2} \dt 
+ (1- \varepsilon ) \nu \int_{0}^T \langle \xi(t), u'(t)- u_1 \rangle \dt  \notag \\
&+ \int_{0}^T \int_{0}^{t} \nu \langle \xi(s), u'(s)- u_1 \rangle \ds \dt 
 + \varepsilon \nu \bar \psi_\lambda(u'(T))+\bphi(u(T)) \notag
\\&+ \varepsilon\nu \int_{0}^T \bar \psi_\lambda(u'(t)) \dt + \int_{0}^T\bphi(u(t)) \dt  \notag \\
&= \varepsilon (T+1 ) \nu\bar \psi_\lambda(u_1) + (T+1) \bphi(u_0) +  \int_{0}^T \langle \eta(t), u_1 \rangle \dt + \int_{0}^T \int_{0}^t \langle \eta(s), u_1 \rangle \ds \dt  .
\end{align}


Using the fact that $\psi, \phi \geq 0$ and $\langle \xi(t), u'(t)- u_1 \rangle  \geq \bar \psi_\lambda(u'(t)) - \bar \psi_\lambda (u_1)$, we get from \eqref{eq:nested}
\begin{align*} 
&    \frac{\rho (4 -  3 \varepsilon) \varepsilon}2  \int_{0}^T \|u''(t)\|^2\dt 
+ \frac{\rho(1 -2\varepsilon) }2 {\|u'(T)-u_1\|^2} 
\\&+  (1 - \varepsilon )\nu \int_{0}^T  \bar \psi_\lambda(u'(t)) \dt  
 + \int_{0}^T\bphi(u(t)) \dt \notag \\
&\leq \Big(\varepsilon + T + \frac{T^2}2 \Big) \nu \psi(u_1) + (T+1)  \bar\phi_\mu (u_0)+  \int_{0}^T \langle \eta(t), u_1 \rangle \dt + \int_{0}^T \int_{0}^t \langle \eta(s), u_1 \rangle \ds \dt  .
\end{align*}
The last two terms in the previous inequality can be treated as follows
\begin{align*}
	&\int_{0}^T \langle \eta(t), u_1 \rangle \dt + \int_{0}^T \int_{0}^t \langle \eta(s), u_1 \rangle \ds \dt
	 = 	\int_{0}^T (1+T-t)\langle \eta(t), u_1 \rangle \dt	\\
	 &\leq (1+T)	\int_{0}^T \| \eta(t)\|_{X^*} |u_1|_X\dt \\
	 &   \leq   C(T) \int_{0}^T  (\| J_\mu u\|_X+\| J_\mu u\|^{r-1}_{L^r(\Omega)} + 1) \|u_1\|_{X} \dt  \\
   &  \leq   \delta  C(T)  \int_{0}^T (\| J_\mu  u\|_X^2+\| J_\mu u\|^{r}_{L^r(\Omega)} ) \dt 
 + C(T,\delta)   (\|u_1\|^2_{X}+ \|u_1\|^r_{X}  + \|u_1\|_{X}  ) \\
  &  \leq   \delta  C(T)  \int_{0}^T\phi( J_\mu u)\dt 
 + C(T,\delta, u_1)  \\
  &  \leq   \delta  C(T)  \int_{0}^T\bphi(u) \dt 
 + C(T,\delta, u_1)  
\end{align*}
 where we used \eqref{growth condition dphi0}, the Young inequality
 for a  sufficiently  small $\delta>0$ to be fixed below, as
 well as \eqref{eq.growthF} and definition \eqref{phi_lambda}. 
As a consequence, we obtain the following estimate
\begin{align}\label{eq:nestedfinal}
	&    \frac{\rho (4 - 3 \varepsilon) \varepsilon}2  \int_{0}^T \|u''(t)\|^2\dt 
	+ \frac{\rho(1 -2\varepsilon) }2 {\|u'(T)-u_1\|^2}\notag \\
	& +  (1- \varepsilon)\nu \int_{0}^T \bar \psi_\lambda(u'(t)) \dt  
	+ c\int_{0}^T\bphi(u(t)) \dt \notag \\
	&\leq \Big(\varepsilon + T + \frac{T^2}2 \Big) \nu \bar \psi_\lambda(u_1) + (T+1) \bphi(u_0) + C(T,\delta, u_1)   
\end{align}
for some strictly positive constant $c=c(T,\delta)<1$.
For a sufficently small $\varepsilon>0$, from \eqref{eq:nestedfinal} we can deduce
\begin{align}
&\varepsilon \rho \| u'' \|^2_{L^2(0,T; H)} \leq C, \label{bnd_uprimeprime}\\
& \rho {\|u'(T)-u_1\|^2} \leq C, \label{bnd_uprimefixedT}
 \end{align}
  \begin{align}
&\int_0^T \bar \psi_{ \lambda}(u'(t)) \, \mathrm{d}t \leq C ,\label{bnd_psilambda}
\\
  &\int_0^T \bphi(u(t)) \, \mathrm{d}t \leq C .\label{bnd_phi}
  \end{align}
 Hence, using \eqref{coercivity psi}-\eqref{growth condition dpsi} and \eqref{eq.growthF}-\eqref{coercivity phi}  as well as \eqref{psi_lambda} and \eqref{phi_lambda}, we get
  \begin{align}
 	& \| \bar J_\lambda u' \|^p_{L^p(0, T ; V)}  \leq C  , \label{bnd_uprimeT} 
 \\	& \| \d_V \psi(\bar J_\lambda u') \|^{p'}_{L^{p'}(0,T ; V^*)}  \leq C  , \label{bnd_dpsiT}
 \\
 &\| J_\mu u\|_{L^2(0, T ; X)}  \leq C \label{bnd_uX}, \\
  &\| J_\mu u\|_{L^r(0, T ; L^r(\Omega))}  \leq C .
 \end{align}
Recalling that $\partial_H \bar \psi _\lambda (u')= \partial_H \bar \psi (\bar J_\lambda u') \subset \d_V \psi ( \bar J_\lambda u')$, we get 
\begin{align}
	\| \partial_H \bar \psi_\lambda ( u') \|^{p'}_{L^{p'}(0,T ; V^*)}  \leq C  . \label{bnd_partialHpsiT}
\end{align}
From 
\eqref{bnd_uprimeprime}, along with $u'(0)=u_1$ and $u(0) = u_0$, it follows that
 \begin{align} 
&  \sqrt{\rho \varepsilon} \sup_{t \in (0,T)} \| u^\prime(t)\|\leq C , \label{bnd_uprimeinf}\\
&  \sqrt{\rho \varepsilon} \sup_{t \in (0,T)} \|u(t)\|
  \leq C.\label{bnd_supu}
 \end{align}
Furthermore, observe that $ \left(\partial_H \bar \psi_{ \lambda} (u') \right)^\prime= \operatorname{Hes} \psi (\bar J_\lambda u') \left( \bar J_\lambda u'\right )^\prime$ and 
 \[
 \left\|\left( \bar J_\lambda u'\right )^\prime(t)\right\| = \lim_{h \rightarrow 0} \left\| \frac{\bar J_\lambda u'(t+h)- \bar J_\lambda u'(t)}{h}\right\| \leq 
 \lim_{h \rightarrow 0} \left\| \frac{u'(t+h)- \bar  u'(t)}{h}\right\| = \|u''(t)\|
 \]
 for a.a. $t \in (0,T)$.
 Hence, we have
 \begin{align}
  \sqrt{\rho \varepsilon} \|\bar J_\lambda u' \|_{W^{1, 2}(0,T; H )} &\leq C, \label{bnd_Jprime}\\
 \sqrt{\rho \varepsilon} \| \partial_H \bar \psi_\lambda(u'(t)) \|_{W^{1,2}(0,T; H)} &\leq C \lambda^{-1}, \label{bnd_psi_lam} \\
 \sqrt{\rho \varepsilon} \| \partial_H \bar \psi_\lambda(u'(t)) \|_{W^{1,q}(0,T; L^q(\Omega))} &\leq C 
\quad \text{with } q=  \frac{2p}{3p-4}  \in (1,p'], \label{bnd_psi}
 \end{align} as $\partial_H \bar \psi _\lambda (u')= \partial_H \bar \psi (\bar J_\lambda u')$ and due to\eqref{growth condition hespsi}-\eqref{growth condition hespsi2}. 

Arguing as above one can check that
\[
\sup_{t \in (0,T)} \| (J_\mu u)'(t)\|^2 \leq \sup_{t \in (0,T)} \| u'(t)\|^2  \leq 2 \| u_1\|^2 + 2T \int_0^T \| u''(t)\|^2 \dt \leq \frac{C}{\varepsilon \rho}. 
\]
The bound \eqref{bnd_uX} hence ensures that 
\begin{align}\label{new_one}
	\| J_\mu u\|_{W^{1,\infty}(0,T;H)\cap L^2(0,T;X)} \leq \frac{C}{\sqrt{\varepsilon \rho}}. 
\end{align}
Note that $W^{1,\infty}(0,T; H) \cap L^2(0,T;X) \hookrightarrow L^8(0,T;L^4(\Omega)) $.
 Since $\eta \in   \partial_H  \bar\phi(J_\mu u) \subset \partial_{X}\phi_{ X}(J_\mu u)$,
 assumptions \eqref{growth condition dphi0} and bounds \eqref{bnd_uX} and \eqref{new_one} ensures that
\begin{align} 
	\rho \varepsilon \| \eta \|^{2}_{L^{2}(0,T;   X^*)} 
	&\leq \rho \varepsilon C  \Big( \| J_\mu u \|^2_{L^2(0,T;X)} + \int_{0}^T \| J_\mu u \|^{2(r-1)}_{L^r(\Omega)}\dt +1 \Big)  \notag\\
	&\leq \rho \varepsilon C \big( 1 + \| J_\mu u \|^8_{L^8(0,T;L^4(\Omega))}\big)
	 \leq C .\label{bnd_eta}  
\end{align}

  
We close this subsection by deriving additional a priori estimates. A comparison in equation \eqref{euler:lambda} yields
 \begin{align*} 
 \|\rho \varepsilon^2 u''''- 2 \rho \varepsilon u'''\|_{L^{2}(0,T;X^{\ast})} &\leq C (\varepsilon,\rho,T) \lambda^{-1}, \\
 \|\rho \varepsilon^2 u''''- 2 \rho \varepsilon u''' - \varepsilon \xi'  \|_{L^{2}(0,T;X^{\ast}) +L^{p'}(0,T;V^{\ast}) } &\leq  C (\varepsilon,\rho,T) , \\
\|\rho \varepsilon^2 u''''- 2 \rho \varepsilon u'''+ \rho u''  \|_{L^{2}(0,T;X^{\ast}) + L^q(0,T;L^q(\Omega)) +L^{p'}(0,T;V^{\ast}) } &\leq
C (\varepsilon,\rho,T) \text{.} 
\end{align*}

Moreover, we deduce additional regularity for the first two terms in \eqref{euler:lambda}.
Observe that, setting $v:=u'''$, we can rewrite \eqref{euler:lambda} as
\begin{align} \label{ODE}
\rho \varepsilon^2 v'(t) - 2 \rho \varepsilon v(t)  + w(t)
= 0  \quad \text{ in } V^\ast \,,
\end{align}
where $w(t):= \rho u'' (t)   + \eta (t) -\varepsilon  \nu
\xi '(t)+ \nu \xi (t) $ for a.e. $t \in (0,T)$. Note that
we have a uniform bound (only in $\lambda$, not in $\varepsilon$) in
the $L^{q}(0,T;L^q(\Omega)+X^{\ast}) $ norm for $w$, thanks to
\eqref{bnd_uprimeprime}, \eqref{bnd_psi}, and \eqref{bnd_eta}. Hence,
solving \eqref{ODE} for $v$ we deduce 
\begin{align*}
\| v\|_{ L^{q}(0,T;L^q(\Omega)+X^{\ast}) } \leq 
C( \varepsilon,\rho,T)\,, \quad  \|v\|_{L^2(0,T;X^{\ast})} \leq 
C(\varepsilon, \rho,T) \lambda^{-1}\,,
\end{align*}
then, by comparison,
\begin{align*}
\| v'\|_{L^{q}(0,T;L^q(\Omega)+X^{\ast}) } \leq  
C( \varepsilon,\rho,T)\,, \quad  \|v'\|_{L^2(0,T;X^{\ast})} \leq  
C( \varepsilon,\rho,T) \lambda^{-1}\,,
\end{align*}
which eventually ensures that
\begin{align*}
\| u '''\|_{ L^{q}(0,T;L^q(\Omega)+X^{\ast}) }  + 	\| u ''''\|_{ L^{q}(0,T;L^q(\Omega)+X^{\ast})} &\leq C(\varepsilon, \rho,T)\,,\\
 \| u '''\|_{L^2(0,T;X^{\ast})} + \|u''''\|_{L^2(0,T;X^{\ast})} &\leq C(\varepsilon, \rho,T) \lambda^{-1}.
\end{align*}

\subsection{Passage to the limit as $\mu \rightarrow 0$}\label{sec:limitmu}

Let $u_{\varepsilon \lambda \mu}$ be a minimizer of $I^{\lambda\mu}_{\rho\varepsilon}$, $\eta_{\varepsilon \lambda \mu} = \partial_H \bphi (u_{\varepsilon \lambda \mu} )$, and $\xi_{\varepsilon \lambda \mu}=\partial_H \bar \psi_{\lambda} ( u_{\varepsilon\lambda\mu}')$. We have proved that $(u_{\varepsilon \lambda \mu},\eta_{\varepsilon \lambda \mu} ,  \xi_{\varepsilon \lambda \mu})$ solves \eqref{euler:lambda}--\eqref{euler:4}. From the uniform estimates of Subsection \ref{subsec_a priori estimates}, we deduce the following convergences as $\mu \rightarrow0_+$ (up to not relabeled subsequences)
\begin{align}
u_{\varepsilon\lambda\mu} & \rightarrow u_{\varepsilon\lambda} \text{ weakly in }W^{4,2}(0,T;X^{\ast}) \cap H^{2}(0,T;H)
,\label{convulambda0}\\
J_\mu u_{\varepsilon \lambda \mu} & \rightarrow v_{\varepsilon \lambda} \text{ weakly in }L^{2}(0,T;X) \text{,} \label{convulambdax0}\\
\xi_{\varepsilon \lambda \mu}  &  \rightarrow\xi_{\varepsilon \lambda} \text{ weakly in } W^{1,2}(0,T;H) \text{,}\label{convxilambda0}\\
\eta_{\varepsilon\lambda\mu}  &  \rightarrow\eta_{\varepsilon\lambda} \text{ weakly in }L^{2}\left(0,T;X^{\ast}\right).\label{convetalambda0}
\end{align}
Convergences \eqref{convulambda0}-\eqref{convetalambda0} are sufficient in order to pass to the limit in equation \eqref{euler:lambda} and obtain 
\begin{equation}
\rho \varepsilon^2 u''''_{\varepsilon\lambda} -2 \rho \varepsilon u'''_{\varepsilon\lambda} + \rho u''_{\varepsilon\lambda} - \varepsilon \xi_{\varepsilon\lambda}' + \xi_{\varepsilon\lambda}+\eta_{\varepsilon\lambda}=0 \ \ \text{ in } X^{\ast}, \text{ a.e. in }  (0,T) \text{.}
\label{eqlimitlambda0}
\end{equation}
Since $X$ is compactly embedded in $H$, we deduce from \eqref{bnd_phi} and \eqref{new_one} that
\begin{align}
&J_\mu u_{\varepsilon \lambda \mu} \rightarrow u_{\varepsilon \lambda} \text{ strongly in } C([0,T];H),\\
&u_{\varepsilon \lambda \mu} \rightarrow u_{\varepsilon\lambda} \text{ strongly in } L^2(0,T;H) \,,\label{convstrongu0}  \\
&u_{\varepsilon \lambda \mu} (t)  \rightarrow u_{\varepsilon\lambda}(t) \text{ strongly in } H \text{ for a.a. } t \in (0,T)\,. \label{convstrongut0}
\end{align}
Moreover, since $H \subset X^{\ast}$ compactly, we have
\begin{align}
\xi_{\varepsilon \lambda \mu}  &\rightarrow \xi_{\varepsilon \lambda} \text{ strongly in }C([0,T];  X^{\ast}), \label{convxilambdax0} \\
u''_{\varepsilon \lambda \mu}  &\rightarrow u''_{\varepsilon \lambda} \text{ strongly in }C([0,T];  X^{\ast}), \label{convu''lambdax0}
\end{align}
hence 
\begin{equation}\label{euler:3l}
u''_{\varepsilon \lambda}(T)=0.
\end{equation}
Furthermore, for all $t \in [0,T]$, we can take a subsequence $\mu^t_n \rightarrow 0$ (possibly depending on $t$) such that
\begin{align}
u'_{\varepsilon \lambda \mu^t_n}(t) &  \rightarrow u'_{\varepsilon \lambda }(t)  \text{ weakly in }  H. \label{convweaku'0}
\end{align}

Eventually, the initial data \eqref{euler:0}-\eqref{euler:1} and the final datum \eqref{euler:3} can be recovered in the limit $\mu \rightarrow 0_+$ thanks to convergence \eqref{convulambda0}. The final datum \eqref{euler:4} can be recovered in the limit arguing as it follows. From \eqref{euler:4} and \eqref{convxilambdax0} one can deduce that
\begin{align*}
\varepsilon \rho u'''_{\varepsilon \lambda \mu} (T)= \xi_{\varepsilon \lambda \mu}(T) &\rightarrow \xi_{\varepsilon\lambda}(T) \text{ strongly in } X^{\ast}.
\end{align*}
On the other hand, it follows from \eqref{convulambda0} that 
\begin{align}\label{aaa}
\varepsilon \rho u'''_{\varepsilon \lambda \mu} (T) \rightarrow \varepsilon \rho u'''_{\varepsilon \lambda} (T)\text{ weakly in } X^{\ast}
\end{align}
(see Appendix \S \ref{apdx}), hence 
\begin{equation}\label{euler:4l}
 \varepsilon \rho u'''_{\varepsilon \lambda} (T) = \xi_{\varepsilon \lambda}(T)  \ \mbox{ in } X^{\ast}.
\end{equation}

\subsubsection{Identification of the nonlinearities}\label{Ss:ident}

\paragraph{\it Identification of $\eta_{\varepsilon}$}

Define the operator $A:X \to X^*$ and recall $ B : X \to X^\ast$ as
\begin{align}
\langle Au,v \rangle_X :=\int_\Omega \nabla u {\cdot} \nabla v \,\d x,
\quad \langle  B(u),v \rangle_X := \int_\Omega f(u) v \,\d x \label{defAB}
\end{align}
so that $\langle Au,u \rangle = \| u \|_X^2$ and $B(u)=f(u)$ almost everywhere. At first, convergence \eqref{convulambdax0} entails that $AJ_\mu u_{\epsi\lambda\mu} \to Au_{\epsi\lambda}$ weakly in $L^{2}(0,T;X^*) $. On the other hand, the strong convergence \eqref{convstrongut0} and the continuity of $f$ ensure  that $f(u_{\epsi\lambda\mu}) \to f(u_{\epsi\lambda})$ almost everywhere. As $f$ is bounded by \eqref{eq.growthF}, we readily get that $f(u_{\epsi\lambda\mu}) \to f(u_{\epsi\lambda})$ strongly in $L^s(0,T;L^s(\Omega))$ for any $s<r'$. As $r \leq p< 2^*$, this in particular implies that $\eta_{\varepsilon\lambda}  = - \Delta u_{\varepsilon\lambda} + f(u_{\varepsilon\lambda})$ in $X^*$ almost everywhere in $(0,T)$. 

\paragraph{\it Identification of $\xi_{\varepsilon\lambda}$}
Recall that $u_{\lambda \varepsilon \mu} \in W^{4,2}(0,T;H)$ is such that $\rho u_{\lambda \varepsilon \mu}''(T)=0$, $u_{\lambda \varepsilon \mu}(0)=u_0$, $\rho u_{\lambda \varepsilon\mu}'(0)=\rho u_1$, and $\xi_{\varepsilon \lambda \mu} \in W^{1,2}(0,T;H)$. Moreover, note that we have $\rho \varepsilon^2 u_{\lambda \varepsilon \mu}''''-2\rho \varepsilon u_{\lambda \varepsilon \mu}'''-\varepsilon \xi'_{\lambda \varepsilon \mu} \in L^{2}(0,T;H)$. Then, we have 
\begin{align*}
\int_{0}^{T}\langle \varepsilon \nu \xi_{\varepsilon \lambda \mu}(t), u'_{\varepsilon \lambda \mu}(t)\rangle \, \mathrm{d}t
&= \langle \varepsilon \nu\xi_{\varepsilon \lambda \mu}(T), u_{\varepsilon \lambda \mu}(T) - u_0\rangle - \int_{0}^{T}\langle \varepsilon \nu \xi'_{\varepsilon \lambda \mu}(t), u_{\varepsilon \lambda \mu}(t) - u_0\rangle \, \mathrm{d}t,
\end{align*}
where we integrated by parts. As a next step, we exploit \eqref{euler:lambda}, obtaining
\begin{align*}
\int_{0}^{T}\langle \varepsilon \nu \xi_{\varepsilon \lambda \mu}(t), u'_{\varepsilon \lambda \mu}(t)\rangle \, \mathrm{d}t 
&= \langle \varepsilon \nu\xi_{\varepsilon \lambda \mu}(T), u_{\varepsilon \lambda \mu}(T) - u_0\rangle  \\
&\quad - \int_{0}^{T}\langle \rho \varepsilon^2 \nu u''''_{\varepsilon \lambda \mu}(t) -2 \rho \varepsilon u'''_{\varepsilon \lambda \mu}(t) , u_{\varepsilon \lambda \mu}(t) - u_0\rangle \, \mathrm{d}t \\
& \quad - \int_{0}^{T}\langle \rho  u''_{\varepsilon \lambda \mu}(t)+  \xi_{\varepsilon \lambda \mu}(t) + \eta_{\varepsilon \lambda \mu}(t) , u_{\varepsilon \lambda \mu}(t) - u_0\rangle \, \mathrm{d}t.
\end{align*}
We integrate by parts the term
\begin{align*}
&\langle \varepsilon \nu\xi_{\varepsilon \lambda \mu}(T), u_{\varepsilon \lambda \mu}(T) - u_0\rangle - \int_{0}^{T}\langle \rho \varepsilon^2 \nu u''''_{\varepsilon \lambda \mu}(t) -2 \rho \varepsilon u'''_{\varepsilon \lambda \mu}(t) , u_{\varepsilon \lambda \mu}(t) - u_0\rangle \, \mathrm{d}t \\
&= \langle \varepsilon \nu\xi_{\varepsilon \lambda \mu}(T) - \rho \varepsilon^2  u'''_{\varepsilon \lambda \mu}(T), u_{\varepsilon \lambda \mu}(T) - u_0\rangle 
+ \langle  2\rho \varepsilon  u''_{\varepsilon \lambda \mu}(T), u_{\varepsilon \lambda \mu}(T) - u_0 \rangle \\
& \quad + \int_{0}^{T}\langle \rho  \varepsilon^2  \nu u'''_{\varepsilon \lambda \mu}(t)  , u_{\varepsilon \lambda \mu}'(t)\rangle \, \mathrm{d}t  
- \int_{0}^{T}\langle 2 \rho \varepsilon u''_{\varepsilon \lambda \mu}(t), u_{\varepsilon \lambda \mu}'(t)\rangle \, \mathrm{d}t \\
&= - \langle  \rho \varepsilon^2 u''_{\varepsilon \lambda \mu}(0) , u_1 \rangle - \int_{0}^{T} \rho  \varepsilon^2  \| u''_{\varepsilon \lambda \mu}(t) \|^2\, \mathrm{d}t - \rho \varepsilon \|u'_{\varepsilon \lambda \mu}(T) \|^2 + \rho  \varepsilon \|u_1 \|^2 ,
\end{align*}
where we used \eqref{euler:3} and \eqref{euler:4}. It follows that
\begin{align} \label{eq:mollificationargument0}
&\int_{0}^{T}\langle \varepsilon \nu \xi_{\varepsilon \lambda \mu}(t), u'_{\varepsilon \lambda \mu}(t)\rangle \, \mathrm{d}t \notag\\
&= -\langle  \rho \varepsilon^2 u''_{\varepsilon \lambda \mu}(0)  , u_1  \rangle - \int_{0}^{T} \rho \varepsilon^2 \| u''_{\varepsilon \lambda \mu}(t)\|^2 \, \mathrm{d}t \notag
\\
& \quad - \rho \varepsilon \|u'_{\varepsilon \lambda \mu}(T)\|^2 + \rho \varepsilon \|u_1\|^2 - \int_0^T \langle  \rho u''_{\varepsilon \lambda}(t) , u_{\varepsilon \lambda \mu}(t) - u_0 \rangle   \, \mathrm{d}t \notag \\
& \quad - \int_0^T \langle \nu \xi_{\varepsilon \lambda} (t), u_{\varepsilon \lambda \mu}(t) - u_0 \rangle    \, \mathrm{d}t - \int_0^T \langle \eta_{\varepsilon \lambda \mu}(t), u_{\varepsilon \lambda \mu}(t)  - u_0\rangle   \, \mathrm{d}t .
\end{align}

Observe now that
\begin{align*}
\int_0^T \langle \eta_{\varepsilon\lambda\mu}(t), J_\mu u_{\varepsilon\lambda \mu}(t) - u_0 \rangle \, \mathrm{d}t 
&= \int_0^T \langle \eta_{\varepsilon\lambda\mu}(t), u_{\varepsilon\lambda\mu}(t) - u_0 \rangle \, \mathrm{d}t - \mu \int_0^T \| \eta_{\varepsilon\lambda \mu}(t)\|^{2} \, \mathrm{d}t \\
& \leq \int_0^T \langle \eta_{\varepsilon\lambda\mu}(t), u_{\varepsilon\lambda\mu}(t) - u_0 \rangle \, \mathrm{d}t ,
\end{align*}
so that we have
\begin{align*}
	 \limsup_{\mu \rightarrow 0_+} \bigg(- \int_0^T \langle \eta_{\varepsilon\lambda\mu}(t),  u_{\varepsilon\lambda\mu}(t) - u_0 \rangle \, \mathrm{d}t \bigg)
	&\leq
 \limsup_{\mu \rightarrow 0_+} \bigg(- \int_0^T \langle \eta_{\varepsilon\lambda\mu}(t), J_\mu u_{\varepsilon\lambda\mu}(t) - u_0 \rangle \, \mathrm{d}t \bigg) \\
&\leq  - \int_0^T \langle \eta_{\varepsilon\lambda}(t), u_{\varepsilon\lambda}(t) - u_0 \rangle_X \, \mathrm{d}t . 
\end{align*}
Moreover, we derive from \eqref{convxilambda0} and \eqref{convstrongu0} that 
\begin{align*}
\int_0^T \langle \xi_{\varepsilon\lambda\mu}(t), u_{\varepsilon\lambda\mu}(t) - u_0 \rangle \, \mathrm{d}t \to \int_0^T \langle \xi_{\varepsilon \lambda}(t), u_{\varepsilon \lambda}(t) - u_0 \rangle \, \mathrm{d}t . 
\end{align*}

Using convergences \eqref{convu''lambdax0}, \eqref{convulambda0}, and \eqref{convweaku'0}, we have
\begin{align*}
\lim_{\mu \rightarrow 0_+} \left(-\langle \rho \varepsilon^2 u''_{\varepsilon \lambda \mu}(0) , u_1 \rangle_X \right) &= -\langle \rho \varepsilon^2 u''_{\varepsilon \lambda}(0) , u_1 \rangle_X,\\
\limsup_{\mu \rightarrow 0_+} \bigg(- \int_{0}^{T} \rho \varepsilon^2 \| u''_{\varepsilon \lambda \mu}(t)\|^2 \, \mathrm{d}t  \bigg) &\leq - \int_{0}^{T} \rho \varepsilon^2 \| u''_{\varepsilon \lambda}(t)\|^2   \, \mathrm{d}t,  \\
\limsup_{\mu \rightarrow 0_+} \big( - \rho \varepsilon \|u'_{\varepsilon \lambda \mu}(T)\|^2\big) &\leq - \rho \varepsilon \|u'_{\varepsilon \lambda}(T)\|^2 .
\end{align*}
Hence, from convergences \eqref{convulambda0}, \eqref{convstrongu0}, and \eqref{convxilambda0}, we can deduce that
\begin{align*}
& \limsup_{\mu \rightarrow 0_+} \int_{0}^{T} \langle \varepsilon \nu \xi_{\varepsilon \lambda \mu}(t), u'_{\varepsilon \lambda \mu}(t)  \rangle   \, \mathrm{d}t \\
&\leq - \langle \rho \varepsilon^2 u''_{\varepsilon \lambda}(0) , u_1 \rangle_X  - \int_{0}^{T} \rho \varepsilon^2 \| u''_{\varepsilon \lambda}(t)\|^2\, \mathrm{d}t - \rho \varepsilon \|u'_{\varepsilon \lambda}(T)\|^2 + \rho \varepsilon \|u_1\|^2 \\
& \quad - \int_{0}^{T} \langle \rho u''_{\varepsilon \lambda} , u_{\varepsilon \lambda}(t) - u_0 \rangle   \, \mathrm{d}t - \int_{0}^{T} \langle \nu  \xi_{\varepsilon \lambda}, u_{\varepsilon \lambda}(t) - u_0\rangle   \, \mathrm{d}t - \int_{0}^{T} \langle \eta_{\varepsilon \lambda}, u_{\varepsilon \lambda}(t) - u_0\rangle_X   \, \mathrm{d}t .
\end{align*}
Integrating by parts back, we then obtain 
\begin{align*}
\limsup_{\mu \rightarrow 0_+} \int_{0}^{T} \langle \varepsilon \nu \xi_{\varepsilon \lambda \mu}(t), u'_{\varepsilon \lambda \mu}(t)\rangle \, \mathrm{d}t 
\leq \int_{0}^{T} \langle \varepsilon \nu \xi_{\varepsilon \lambda}(t), u'_{\varepsilon \lambda}(t)\rangle \, \mathrm{d}t .
 \end{align*}

Recall that $\xi_{\varepsilon \lambda \mu} = \partial_H \bar \psi_\lambda(u_{\varepsilon \lambda \mu}')$. From the demiclosedness of the maximal monotone operator $u \mapsto  \partial_H \bar \psi_\lambda (u)$ in $L^2(0,T;H) \times L^2(0,T;H)$, we can conclude by Lemma \ref{lemma:doublelimsup} that $\xi_{\varepsilon\lambda}(t)$ coincides with $\partial_H \bar \psi_\lambda  (u'_{\varepsilon\lambda}(t))$ for a.a. $t \in (0,T)$ and that
\begin{align}
 \lim_{\mu \rightarrow 0_+} \int_{0}^{T} \langle \xi_{\varepsilon \lambda \mu}(t), u'_{\varepsilon \lambda \mu}(t)- u_1 \rangle \dt  = \int_{0}^{T} \langle \xi_{\varepsilon \lambda}(t), u'_{\varepsilon \lambda}(t)- u_1 \rangle \dt . \label{eqidenxieps0}
 \end{align}

\subsection{Passage to the limit as $\lambda \rightarrow 0$}\label{sec:limitlambda}
From the uniform estimates of Subsection \ref{subsec_a priori estimates}, 
we deduce the following convergences as $\lambda \rightarrow0$ (up to not relabeled subsequences)
\begin{align}
u_{\varepsilon\lambda }  &  \rightarrow u_ \varepsilon \text{ weakly in }  W^{4,q}(0,T;L^q(\Omega)+ X^{\ast})  \cap H^{2}(0,T;H)  
,\label{convulambda} 
\\
u_{\varepsilon \lambda }  &  \rightarrow  u_\varepsilon  \text{ weakly in }L^{2}(0,T;X) \text{,} \label{convulambdax}\\
 \bar J_\lambda  u^\prime_{\varepsilon \lambda }  &  \rightarrow w_\varepsilon \text{ weakly}^{\star} \text{ in } W^{1,2}(0,T;H)\cap L^{p}(0,T;V)  \text{,} \label{convuprimelambda}
\\
\xi_{\varepsilon \lambda }  &  \rightarrow\xi_\varepsilon \text{ weakly in }  W^{1,q}(0,T;L^q(\Omega)) \cap L^{p'
}(0,T;V^{\ast}) \text{,}\label{convxilambda}\\
\eta_{\varepsilon\lambda}  &  \rightarrow\eta_\varepsilon \text{ weakly in }L^{2}\left(
0,T;X^{\ast}\right)   .\label{convetalambda}
\end{align}
Convergences \eqref{convulambda}-\eqref{convetalambda} are sufficient in order to pass to the limit in equation \eqref{eqlimitlambda0} and obtain 
\begin{equation}
\rho \varepsilon^2 u''''_\varepsilon -2 \rho \varepsilon u'''_\varepsilon + \rho u''_\varepsilon - \varepsilon \xi_\varepsilon' + \xi_\varepsilon+\eta_\varepsilon=0 \ \ \text{ in }  L^q(\Omega)+X^{\ast}  \text{ for a.e. }  t \in (0,T) \text{.}
\label{eqlimitlambda}
\end{equation}
%
Moreover, it follows from \eqref{bnd_psilambda} and \eqref{convulambda} that $w_\varepsilon= u^\prime_{\varepsilon}.$ 
Since $X$ is compactly embedded in $H$, exploiting Ascoli's lemma,  we can deduce that 
\begin{align}
u_{\varepsilon \lambda }  &  \rightarrow u_\varepsilon \text{ strongly in } C([0,T];H), \label{convstrongu}
\end{align}
and up to a (not relabeled) subsequence  
\begin{align}
u_{\varepsilon \lambda } (t)  \rightarrow u_\varepsilon(t) \text{ strongly in } H \text{ for all }  t \in (0,T)\,. \label{convstrongut}
\end{align}
Furthermore, by interpolation we have that
$$
\|u\|_V \leq \|u\|^{\theta} \|u\|_X^{1-\theta}
$$
for $\theta = 1- d (\frac12 - \frac1p )$. Recalling that $u_{\varepsilon \lambda}$ is bounded in $L^2(0,T; X)$ and in $W^{1,2}(0,T; H)$, we can apply \cite[Corollary 8]{Simon} and, since $p <4$ and $d \in \{2,3\}$, we find that 
\begin{align}
u_{\varepsilon \lambda} \to u_{\varepsilon} \quad \mbox{ strongly in } L^p(0,T;V).\label{c:u:CV} 
\end{align}


Since $V^{\ast} \hookrightarrow X^{\ast}$ compactly and $L^2(\Omega) \hookrightarrow  X^{\ast} + L^q(\Omega)$ compactly, using an Aubin-Lions compactness lemma, we have 
\begin{align}
\xi_{\varepsilon \lambda }  &\rightarrow \xi_\varepsilon \text{ strongly in } L^{p'}(0,T; X^{\ast}) \cap C([0,T];X^*+L^q(\Omega)),  \label{convxilambdax} \\
u''_{\varepsilon \lambda }  &\rightarrow u''_\varepsilon \text{ strongly in } C([0,T]; X^{\ast}  + L^q(\Omega)), \label{convu''lambdax}
\end{align}
hence $u''_\varepsilon(T)=0$ from \eqref{euler:3l}. 
Furthermore, note that $\liminf_{\lambda \rightarrow 0} \|u'_{\varepsilon \lambda }(t)\| < \infty$ for all $t \in [0,T]$ due to \eqref{bnd_uprimeinf}. Hence we can take a subsequence $\lambda^t_n \rightarrow 0$ (possibly depending on $t$) such that
\begin{align}
u'_{\varepsilon \lambda }(t) &  \rightarrow u'_{\varepsilon }(t)  \text{ weakly in }  H. \label{convweaku'}
\end{align}

Finally, the initial data \eqref{euler:0}-\eqref{euler:1} and the final datum \eqref{euler:3} can be recovered in the limit $\lambda \rightarrow 0$ thanks to convergence \eqref{convulambda}.
The final datum \eqref{euler:4} can be recovered in the limit arguing as it follows. 
From \eqref{euler:4} and \eqref{convxilambdax} one can deduce that
\begin{align*}
\varepsilon \rho u'''_{\varepsilon \lambda } (T)= \xi_{\varepsilon \lambda }(T) &\rightarrow \xi_\varepsilon(T) \text{ strongly in } X^{\ast} + L^q(\Omega).
\end{align*}
On the other hand, it follows from \eqref{convulambda} that 
\begin{align}\label{bbb}
\varepsilon \rho u'''_{\varepsilon \lambda } (T) \rightarrow 	\varepsilon \rho u'''_{\varepsilon } (T)\text{ weakly in } L^q(\Omega)+X^{\ast}
\end{align}
 (see Appendix \S \ref{apdx}), hence $\varepsilon \rho u'''_{\varepsilon } (T) = \xi_{\varepsilon }(T) $ in $L^q(\Omega)+X^{\ast}$.

\subsubsection{Identification of the nonlinearities}

\paragraph{\it Identification of $\eta_{\varepsilon}$}

As in \S \ref{Ss:ident}, we can prove that
$\eta_{\varepsilon}  = - \Delta u_{\varepsilon}  + f(u_{\varepsilon} )
$ in $X^*$ almost everywhere in $(0,T)$.

\paragraph{\it Identification of $\xi_{\varepsilon}$}

Recall that $u_{\lambda \varepsilon} \in W^{4,2}(0,T; X^*) \cap H^{2}(0,T;L^2(\Omega))$ is such that $\rho u_{\lambda \varepsilon}''(T)=0$, $u_{\lambda \varepsilon}(0)=u_0$, $\rho u_{\lambda \varepsilon}'(0)=\rho u_1$,
and $\xi_{\varepsilon \lambda } \in W^{1,2
}(0,T;H) \cap L^{p'}(0,T;V^{\ast})$. Moreover, note that we have $\rho \varepsilon^2 u_{\lambda \varepsilon}''''-2\rho \varepsilon u_{\lambda \varepsilon}'''-\varepsilon \xi'_{\lambda \varepsilon} \in L^{2}(0,T;X^*)$.
Then, we have 
\begin{align*}
	\int_{0}^{T}\langle \varepsilon \nu \xi_{\varepsilon \lambda }(t), u'_{\varepsilon \lambda }(t)\rangle \, \mathrm{d}t 
	&= 	\langle \varepsilon \nu\xi_{\varepsilon \lambda }(T), u_{\varepsilon \lambda }(T) - u_0\rangle  - \int_{0}^{T}\langle \varepsilon \nu \xi'_{\varepsilon \lambda }(t), u_{\varepsilon \lambda }(t) - u_0\rangle \, \mathrm{d}t,
\end{align*}
where we integrated by parts. 

As a next step, we exploit \eqref{euler:lambda}, obtaining
\begin{align*}
\int_{0}^{T}\langle \varepsilon \nu \xi_{\varepsilon \lambda }(t), u'_{\varepsilon \lambda }(t)\rangle \, \mathrm{d}t 
	&= 	\langle \varepsilon \nu\xi_{\varepsilon \lambda }(T), u_{\varepsilon \lambda }(T) - u_0\rangle  \\
	&\quad  - \int_{0}^{T}\langle \rho \varepsilon^2 \nu u''''_{\varepsilon \lambda }(t) -2 \rho \varepsilon u'''_{\varepsilon \lambda }(t) , u_{\varepsilon \lambda }(t) - u_0\rangle_X \, \mathrm{d}t \\
	& \quad - \int_{0}^{T}\langle  \rho  u''_{\varepsilon \lambda }(t)+  \xi_{\varepsilon \lambda } + \eta_{\varepsilon \lambda } , u_{\varepsilon \lambda }(t) - u_0\rangle_X \, \mathrm{d}t .
\end{align*}
We integrate by parts the term
\begin{align*}
&\langle \varepsilon \nu\xi_{\varepsilon \lambda }(T), u_{\varepsilon \lambda }(T) - u_0\rangle 
- \int_{0}^{T}\langle \rho \varepsilon^2 \nu u''''_{\varepsilon \lambda }(t) -2 \rho \varepsilon u'''_{\varepsilon \lambda }(t) , u_{\varepsilon \lambda }(t) - u_0\rangle_X \, \mathrm{d}t \\
&= \langle \varepsilon \nu\xi_{\varepsilon \lambda }(T) - \rho \varepsilon^2  u'''_{\varepsilon \lambda }(T), u_{\varepsilon \lambda }(T) - u_0\rangle_X 
	 +	\langle  2\rho \varepsilon  u''_{\varepsilon \lambda }(T), u_{\varepsilon \lambda }(T) - u_0 \rangle \\
	& \quad + \int_{0}^{T}\langle \rho  \varepsilon^2  \nu u'''_{\varepsilon \lambda }(t)  , u_{\varepsilon \lambda }'(t)\rangle_X \, \mathrm{d}t  
	- \int_{0}^{T}\langle 2 \rho \varepsilon u''_{\varepsilon \lambda }(t) , u_{\varepsilon \lambda }'(t) \rangle \, \mathrm{d}t \\
	&=    - \langle  \rho \varepsilon^2 u''_{\varepsilon \lambda }(0) , u_1 \rangle_X - \int_{0}^{T} \rho  \varepsilon^2  \| u''_{\varepsilon \lambda }(t) \|^2 \, \mathrm{d}t  
	 - \rho \varepsilon \|u'_{\varepsilon \lambda }(T) \|^2 + \rho  \varepsilon \|u_1 \|^2  ,
\end{align*}
where we used \eqref{euler:3l} and \eqref{euler:4l}. 
It follows that
\begin{align} \label{eq:mollificationargument}
&\int_{0}^{T}\langle \varepsilon \nu \xi_{\varepsilon \lambda }(t), u'_{\varepsilon \lambda }(t)\rangle \, \mathrm{d}t \notag\\
&=
-	\langle  \rho \varepsilon^2 u''_{\varepsilon \lambda }(0)  , u_1  \rangle - \int_{0}^{T}   \rho \varepsilon^2 \| u''_{\varepsilon \lambda }(t)\|^2  \, \mathrm{d}t \notag
\\
& \quad -	 \rho \varepsilon \|u'_{\varepsilon \lambda }(T)\|^2   +	 \rho \varepsilon \|u_1\|^2 	 -	\int_0^T \langle  \rho u''_{\varepsilon \lambda } , u_{\varepsilon \lambda }(t) - u_0 \rangle   \, \mathrm{d}t \notag \\
& \quad  -	\int_0^T \langle \nu \xi_{\varepsilon \lambda }, u_{\varepsilon \lambda }(t) - u_0 \rangle    \, \mathrm{d}t 
- \int_0^T \langle \eta_{\varepsilon \lambda }, u_{\varepsilon \lambda }(t)  - u_0\rangle_X   \, \mathrm{d}t  
.
\end{align}

Observe now that

\begin{align*}
\liminf_{\lambda \rightarrow 0} \int_0^T \langle \eta_{\varepsilon\lambda}(t), u_{\varepsilon\lambda}(t) - u_0 \rangle_X \, \mathrm{d}t 
\geq \int_0^T \langle \eta_{\varepsilon}(t), u_{\varepsilon}(t) - u_0 \rangle_X \, \mathrm{d}t . 
\end{align*}
Furthermore, using \eqref{convxilambda} and \eqref{c:u:CV}, we have that
\begin{align*}
\lim_{\lambda \to 0_+} \int_0^T \langle \xi_{\varepsilon\lambda}(t), u_{\varepsilon\lambda}(t) - u_0 \rangle \, \mathrm{d}t 
\to \int_0^T \langle \xi_{\varepsilon}(t), u_{\varepsilon}(t) - u_0 \rangle \, \mathrm{d}t.
\end{align*}

Using convergences \eqref{convu''lambdax}, \eqref{convulambda}, and \eqref{convweaku'}, we have
\begin{align*}
\lim_{\lambda \rightarrow 0} \left(-	\langle  \rho \varepsilon^2 u''_{\varepsilon \lambda }(0)  , u_1  \rangle_{X \cap L^{q'}(\Omega)}   \right) &=   -	\langle  \rho \varepsilon^2 u''_{\varepsilon  }(0)  , u_1  \rangle_{X \cap L^{q'}(\Omega)}   ,\\
\limsup_{\lambda \rightarrow 0} \bigg(- \int_{0}^{T}   \rho \varepsilon^2 \| u''_{\varepsilon \lambda }(t)\|^2 \, \mathrm{d}t  \bigg) &\leq - \int_{0}^{T}   \rho \varepsilon^2 \| u''_{\varepsilon }(t)\|^2   \, \mathrm{d}t,  \\
\limsup_{\lambda \rightarrow 0} \big( -	 \rho \varepsilon \|u'_{\varepsilon \lambda }(T)\|^2\big) &\leq  - \rho \varepsilon \|u'_{\varepsilon }(T)\|^2 .
\end{align*}
Hence, from convergences \eqref{convulambda},
\eqref{convstrongu} and \eqref{convxilambda}, we can deduce that
\begin{align*}
& \limsup_{\lambda \rightarrow 0} 	\int_{0}^{T} \langle \varepsilon \nu  \xi_{\varepsilon \lambda }(t), u'_{\varepsilon \lambda }(t)  \rangle  \, \mathrm{d}t \\
&\leq  	 
-	\langle  \rho \varepsilon^2 u''_\varepsilon(0)  , u_1 \rangle_{X \cap L^{q'}(\Omega)}  - \int_{0}^{T}   \rho \varepsilon^2 \| u''_{\varepsilon }(t)\|^2  \, \mathrm{d}t -	 \rho \varepsilon \|u'_{\varepsilon }(T)\|^2  +	 \rho \varepsilon \|u_1\|^2 \\
& \quad 	-	\int_{0}^{T} \langle  \rho u''_{\varepsilon  } , u_{\varepsilon  }(t)  - u_0 \rangle  \, \mathrm{d}t  -	\int_{0}^{T} \langle \nu  \xi_{\varepsilon  }, u_{\varepsilon}(t)  - u_0\rangle   \, \mathrm{d}t 
-	\int_{0}^{T} \langle \eta_{\varepsilon }, u_{\varepsilon}(t)  - u_0\rangle_X   \, \mathrm{d}t .
\end{align*}
Integrating back by parts, we then obtain 
 \begin{align*}
&\limsup_{\lambda \rightarrow 0} \int_{0}^{T} \langle \varepsilon \nu \xi_{\varepsilon \lambda }(t), u'_{\varepsilon \lambda }(t)\rangle   \, \mathrm{d}t 
 \leq  \int_{0}^{T} \langle \varepsilon \nu \xi_{\varepsilon }(t), u'_{\varepsilon}(t)\rangle  \, \mathrm{d}t .
 \end{align*} 
 
 Recall that $\xi_{\varepsilon \lambda} = \partial_H \bar \psi_\lambda  (u_{\varepsilon \lambda }')$ and $\partial_H \bar \psi_\lambda  (u'_{\varepsilon \lambda }) \subset \partial_V \bar \psi_\lambda  (u'_{\varepsilon \lambda }) = \mathrm{d}_{V} \bar \psi_\lambda  (u'_{\varepsilon \lambda })$.
From the demiclosedness of the maximal monotone operator $u 
\mapsto  \mathrm{d}_{V} \bar \psi_\lambda  (u)$ in $L^p(0,T;V) \times L^{p'}(0,T;V^\ast)$, we can conclude by Lemma \ref{lemma:doublelimsup} that $\xi_\varepsilon(t)$ coincides with $ \mathrm{d}_{V} \bar \psi_\lambda  (u'_\varepsilon(t))$ for a.a. $t \in (0,T)$, as well as 
 \begin{align}
 \lim_{\lambda \rightarrow 0} \int_{0}^{T} \langle \xi_{\varepsilon \lambda}(t), u'_{\varepsilon \lambda}(t)- u_1 \rangle \dt  = \int_{0}^{T} \langle \xi_{\varepsilon}(t), u'_{\varepsilon }(t)- u_1 \rangle \dt. \label{eqidenxieps}
 \end{align}

\subsection{Minimization of the WIDE functional $I_{\rho\varepsilon}$}\label{sec_minimum}
 Our next aim is to show that the above-determined limit 
$u_{\varepsilon}$ is the unique minimizer of 
$I_{\rho \varepsilon}$  on  $K(u_{0}, u_1)$. 
More precisely, we prove that the strong solution 
$(u_\epsi, \xi_\epsi, \eta_\epsi)$ of the Euler-Lagrange problem 
(\ref{euler1})-(\ref{euler8}) is a minimizer of $I_{\rho\varepsilon}$ in $\mathcal{V}$.

Note that  $ K(u_{0}, u_1) \subset K_\lambda(u_{0}, u_1) 
$, $\bar \psi_{\lambda}\leq\psi$,  and $\bphi\leq\phi$.
By passing to the limit as $\mu \rightarrow0$ and $\lambda\rightarrow0$ and using the
Dominated Convergence Theorem, we have
\[
I_{\rho\varepsilon}^{\lambda\mu}(v)\rightarrow I_{\rho\varepsilon}(v)\text{ \ \ \ }
\forall v\in  K(u_{0}, u_1) \text{.}%
\]
As $u_{\varepsilon\lambda \mu}$ is a global minimizer of $I_{\rho\varepsilon}^{\lambda\mu}$, we have
\[
I_{\rho\varepsilon}^{\lambda\mu}(v)\geq  I_{\rho\varepsilon}^{\lambda\mu}(u_{\varepsilon\lambda \mu})\text{ \ \ \ } \forall v  \in K(u_{0}, u_1) \text{.}%
\]
Convergences (\ref{convulambda0})-(\ref{convulambdax0}) as well as (\ref{convulambda})-(\ref{convulambdax}) together with the lower
semicontinuity of the convex integrals  
$u\mapsto \int_{0}^{T}\mathrm{e}^{-t/\varepsilon
} \frac{\varepsilon^2 \rho}{2}\|u'' (t)\|^2\mathrm{d}t$, $u\mapsto$ $ \int_{0}^{T}\mathrm{e}^{-t/\varepsilon
}\varepsilon \nu \psi(u'(t))\mathrm{d}t$, and $u \mapsto\int_{0}%
^{T}\mathrm{e}^{-t/\varepsilon}\phi\left(  u(t)\right)  \mathrm{d}%
t\mathcal{\ }$in $L^{p}(0,T;V)$ give 
\begin{align*}
\lefteqn{
\liminf_{\lambda\rightarrow0} \liminf_{\mu\rightarrow0} I_{\rho\varepsilon}^{\lambda\mu}(u_{\varepsilon\lambda \mu})
}\\
&= \liminf_{\lambda\rightarrow0} \liminf_{\mu\rightarrow0}  \int_{0}^{T}\mathrm{e}^{-t/\varepsilon
}\Big( \frac{\varepsilon^2 \rho}{2}\|u''_{\varepsilon\lambda\mu}(t)\|^2  + \varepsilon\nu \bar \psi_\lambda (u_{\varepsilon\lambda \mu}'(t))+\bphi\left(  u_{\varepsilon\lambda \mu}(t)\right)  \Big)  \mathrm{d}t\\
&\geq \liminf_{\lambda\rightarrow0} \liminf_{\mu\rightarrow0}  \int_{0}^{T}\mathrm{e}^{-t/\varepsilon
}\Big( \frac{\varepsilon^2 \rho}{2}\|u''_{\varepsilon\lambda\mu}(t)\|^2 +  \varepsilon\nu \bar \psi_\lambda (u_{\varepsilon\lambda\mu}'(t))+\phi\left(
J_{\mu}u_{\varepsilon\lambda\mu}(t)\right)  \Big)  \mathrm{d}t\\
&  \geq \liminf_{\lambda\rightarrow0}\int_{0}^{T}\mathrm{e}^{-t/\varepsilon
}\Big( \frac{\varepsilon^2 \rho}{2}\|u''_{\varepsilon\lambda}(t)\|^2 +  \varepsilon\nu \psi (\bar J_\lambda u_{\varepsilon\lambda}'(t)) +\phi\left(
u_{\varepsilon\lambda}(t)\right)  \Big)  \mathrm{d}t\\
&  \geq\int_{0}^{T}\mathrm{e}^{-t/\varepsilon}\Big(  \frac{\varepsilon^2 \rho}{2}\|u''_{\varepsilon}(t)\|^2  + \varepsilon
\nu \psi(u_{\varepsilon}'(t))+\phi\big(  u_{\varepsilon}(t)\big)
\Big)  \mathrm{d}t\text{.}%
\end{align*}
As a consequence, we have
\[
I_{\rho\varepsilon}(v)\geq I_{\rho \varepsilon}(u_{\varepsilon})\text{ \ \ \
}\forall v\in  K(u_{0}, u_1)  \text{.}%
\]
Namely, $u_{\varepsilon}$ minimizes $I_{\rho\varepsilon}$  on $
K(u_{0}, u_1) $, hence $u_{\varepsilon}\in D( \partial I_{\rho\varepsilon})$ and $0 \in \partial_{\mathcal{V}}I_{\rho\varepsilon} (u_\varepsilon)$.


 \section{The causal limit} \label{sec_causallimit}
In this section, we proceed to the proof of Theorem \ref{main_thm} by
checking that, up to subsequences, $u_\epsi$ converges to a strong
solution of the target problem  \eqref{prob1}-\eqref{prob5}.

Starting from the uniform
estimates derived in Subsection \ref{subsec_a priori estimates}, we deduce the following 
bounds, valid for all $T>0$, 
\begin{align} 
&\| u'_{\varepsilon} \|_{L^p(0, T ; V)}  + \| \xi_{\varepsilon} \|_{L^{p'}(0,T ; V^*)}   + \int_{0}^{T} \psi(u'_{\varepsilon} (t)) \d t  \leq C (T) , \label{eq:causalestimates}
\end{align}
whence
 \begin{align} 
	&\sup_{t \in (0,T)} \|u(t)\|_{V 
	} \leq C\label{bnd_supucausal}.
\end{align}
Furthermore, using the
lower semicontinuity of norms and of $\phi$ and \eqref{growth condition dphi0}, we obtain
\begin{align} 
&\| u_{\varepsilon}  \|_{L^2(0,T; X)}	+ \int_{0}^{T} \phi(u_{\varepsilon} (t)) \d t + \| \eta_{\varepsilon}  \|_{L^{2}(0,T;  X^*)} \notag \\
&+ \|\rho \varepsilon^2 u''''_{\varepsilon} - 2 \rho \varepsilon u'''_{\varepsilon} + \rho u''_{\varepsilon}  - \varepsilon \nu \xi'_{\varepsilon}  \|_{ L^{p'}(0,T;X^{\ast})} \leq C(T ).
\end{align}
Up to not relabeled subsequences, these bounds entail the following
convergences as
$\varepsilon\rightarrow0+:$%
\begin{align}
u_{\varepsilon}  &  \rightarrow u\text{ weakly in }W^{1,p}(0,T;V)\text{ and
	strongly in }  C([0,T]; V),\label{convu} \\
u_{\varepsilon}  &  \rightarrow u\text{ weakly in }L^{2}(0,T;X) \text{,} \label{convux}\\
\xi_{\varepsilon}  &  \rightarrow\xi\text{ weakly in }L^{p'%
}(0,T;V^{\ast})\text{,}\label{c1}\\
\rho \varepsilon^2 u''''_{\varepsilon}- 2 \rho \varepsilon u'''_{\varepsilon}+ \rho u''_{\varepsilon} - \varepsilon\xi_{\varepsilon}'  &  \rightarrow \zeta \text{ weakly in }%
 L^{p'}(0,T;X^{\ast}) \text{,}\label{c2}\\
\eta_{\varepsilon}  &  \rightarrow\eta\text{ weakly in }L^{2}(
0,T;X^{\ast}). \label{c3}
\end{align}
Convergences \eqref{c1}-\eqref{c3} are sufficient in order to pass to the limit in equation \eqref{eqlimitlambda} and obtain 
\begin{equation}
\zeta + \nu \xi+\eta=0 \ \ \text{ in }X^{\ast}\text{ a.e. in } (0,T)\text{.}
\label{eqX*}
\end{equation}


 \subsubsection*{Initial conditions}
 Note that the strong convergence in \eqref{convu} follows from the compact embedding $W^{1,p}(0,T;V) \cap L^{2}(0,T;X)
 \hookrightarrow C([0,T];B) $ for some Banach space $B$ such that
 $X\hookrightarrow B$ compactly  
 and from an application of the Aubin-Lions Lemma. 
 The initial condition $u(0)=u_0$ is a direct consequence of \eqref{convu}.
 
 Furthermore, by defining $w_\varepsilon:=\rho \varepsilon^2 u'''_{\varepsilon} - 2 \rho \varepsilon u''_{\varepsilon} + \rho u'_{\varepsilon}  - \varepsilon \nu \xi_{\varepsilon}$,
 we observe that the uniform bound $\| w_\varepsilon\|_{C([0,T]; X^\ast)} \leq C$ follows from \eqref{eq:causalestimates}. In particular, we can take a subsequence $\epsi_n \rightarrow 0$ such that
 \begin{align*}
 	w_{\varepsilon}(0) &  \rightarrow w_0  \text{ weakly in }  X^\ast.
 \end{align*}
 Finally, by testing $w_{\varepsilon}(0)$ with $v \in X$, integrating by parts and letting $\epsi \rightarrow 0$, we can deduce that $w_0= \rho u_1$. 

  \subsubsection*{Identification of the nonlinearities}
  As a final step, we need to identify the limits $\zeta, \xi$, and $\eta $ in \eqref{eqX*}.
\par{\textit{Identification of the limit \({\zeta\equiv \rho u'' } \).}} 
Let \(\varphi \in C^\infty_c([0,T]\times \Omega )\). By integrations by parts we obtain
\begin{align*}
&\int_{0}^T \langle \rho \varepsilon^2 u''''_{\varepsilon}- 2 \rho \varepsilon u'''_{\varepsilon}+\rho u''_{\varepsilon}- \varepsilon \nu \xi_{\varepsilon}'  , \varphi \rangle_X \d t  \\
& = \int_{0}^T \langle \rho \varepsilon^2 u''_{\varepsilon}   , \varphi'' \rangle \d t + \int_{0}^T \langle 2 \rho \varepsilon u''_{\varepsilon} , \varphi' \rangle \d t - \int_{0}^T \langle  \rho u'_{\varepsilon}  , \varphi' \rangle \d t -
\int_{0}^T \langle  \varepsilon \nu\xi_{\varepsilon}  , \varphi ' \rangle  \d t \\
&\rightarrow - \int_{0}^T \langle  \rho u'  , \varphi' \rangle_{X }\d t  =  \int_{0}^T \langle  \rho u'' , \varphi \rangle_{X }\d t\text{ \ \ as } \varepsilon \rightarrow 0,
\end{align*}
where we used \eqref{bnd_uprimeprime}, \eqref{bnd_dpsiT}, and \eqref{convu} to pass to the limit. Hence, $\zeta \equiv \rho u''$ almost everywhere.

\par{\textit{Identification of the limit \({\eta \equiv - \Delta u + f(u)}\).}}
Arguing as in Subsection \ref{sec:limitlambda}, we exploit the weak convergence \eqref{convux} and the a.e. pointwise convergence \eqref{convu} in order to obtain that $\eta_\epsi = - \Delta u_\epsi + f(u_\epsi) \to   - \Delta u + f(u) $ almost everywhere in $X^*$. Hence, $\eta \equiv - \Delta u + f(u) $ in $X^\ast$.

\par{\textit{Identification of the limit \({\xi \equiv  \mathrm{d}_{V}\psi(u') }\).}}
Integrating \eqref{eq:nestedt} over $(0,T)$ and integrating by parts, we obtain 
\begin{align} \notag 
&\int_{0}^{T} \int_{0}^{t} \langle \nu \xi_{\varepsilon \lambda  \mu}(s), u'_{\varepsilon \lambda  \mu}(s)- u_1 \rangle \ds \dt \\
&\leq 
\frac{\rho \varepsilon^2}{2}\int_0^T \|u_{\varepsilon \lambda  \mu}''(t)\|^2 \dt + \rho \varepsilon  \|u'_{\varepsilon \lambda  \mu}(T)- u_1\|^2 -\frac{\rho }{2} \int_0^T \|u'_{\varepsilon \lambda \mu}(t)-u_1\|^2 \dt  + \varepsilon T \nu \psi(u_1) \notag \\ &
\quad + \int_0^T \varepsilon  \nu \langle \xi_{\varepsilon \lambda \mu}(t), u'_{\varepsilon \lambda \mu}(t)-u_1 \rangle \dt  
-\int_0^T \bphi(u_{\varepsilon \lambda \mu}(t))\dt  + T\phi(u_0) + \int_{0}^T  \int_0^t \langle \eta_{\varepsilon \lambda  \mu}(s), u_1 \rangle \ds \dt \notag \\
& \leq \varepsilon C -\frac{\rho }{2} \int_0^T \|u'_{\varepsilon \lambda \mu}(t)-u_1\|^2 \dt  
-\int_0^T \phi( J_\mu u_{\varepsilon \lambda \mu}(t))\dt  + T\phi(u_0) + \int_{0}^T  \int_0^t \langle \eta_{\varepsilon \lambda \mu}(s), u_1 \rangle \ds \dt \notag 
\end{align}
due to \eqref{bnd_uprimeprime}, \eqref{bnd_uprimefixedT}, \eqref{bnd_uprimeT}, and \eqref{bnd_dpsiT}. Then, we take the $\limsup$ of both sides as $\mu \rightarrow 0$ and then as ${\lambda \rightarrow 0}$.
On the left-hand side we pass to the limit thanks to \eqref{eqidenxieps0} and \eqref{eqidenxieps}, while on the right-hand side we use \eqref{convulambda0}, \eqref{convetalambda0}, \eqref{convulambda} and \eqref{convetalambda} together with 
\[
\liminf_{\lambda \rightarrow 0}  \liminf_{\mu \to 0} \phi ( J_\mu  u_{\varepsilon \lambda \mu}) \geq \liminf_{\lambda \rightarrow 0} \phi (u_{\varepsilon \lambda}) \geq \phi ( u_{\varepsilon}) \,,
\]
which is given by the weak lower continuity of $\phi$ in $X$, \eqref{convulambdax0} with $v_\varepsilon=u_\varepsilon$ and \eqref{convulambdax}. In particular, we get 
\begin{align*} 
&\int_{0}^{T} \int_{0}^{t} \langle \nu \xi_{\varepsilon }(s), u'_{\varepsilon}(s)- u_1 \rangle \ds \dt \\
& \leq  -\frac{\rho }{2} \int_0^T \|u'_{\varepsilon}(t)-u_1\|^2 \dt  
-\int_0^T \phi(u_{\varepsilon }(t))\dt  + T\phi(u_0) + \int_{0}^T  \int_0^t \langle \eta_{\varepsilon }(s), u_1 \rangle \ds \dt \notag 	.
\end{align*}
By arguing as above, convergences \eqref{convu}, \eqref{convux}, and \eqref{c3} give 
 \begin{align} \label{ineqcausal}
 &\limsup_{\varepsilon \rightarrow 0} \int_{0}^{T} \int_{0}^{t} \nu \langle \xi_{\varepsilon }(s), u'_{\varepsilon}(s)- u_1 \rangle \ds \dt \\
 & \leq  -\frac{\rho }{2} \int_0^T \|u'(t)-u_1\|^2 \dt  
 -\int_0^T \phi(u(t))\dt  + T\phi(u_0) + \int_{0}^T  \int_0^t \langle \eta(s), u_1 \rangle_X \ds \dt \notag  	\\
 & \leq- \int_{0}^T  \int_0^t  \langle \rho u''(s) + \eta(s), u'(s)- u_1 \rangle_X \ds \dt = \int_{0}^T  \int_0^t \nu\langle  \xi(s), u'(s)- u_1 \rangle \ds \dt \notag 
 \end{align}
 where we integrated by parts and we used \eqref{eqX*}.
  In particular, recalling Remark \ref{rem:limsup}, we have 
  \begin{align} 
 	\limsup_{\varepsilon \rightarrow 0} \int_{0}^{T}(T-t) \nu  \langle \xi_{\varepsilon }(s), u'_{\varepsilon}(s)- u_1 \rangle  \dt 
 	\leq \int_{0}^T  (T-t) \nu\langle  \xi(s), u'(s)- u_1 \rangle  \dt \notag .
 \end{align}
From the demiclosedness of the maximal monotone operator $u 
\mapsto \mathrm{d}_V \psi(u(\cdot))$ in  ${\mathcal V} \times
{\mathcal V}^*$  we can conclude by Lemma \ref{lemma:doublelimsup} that $\xi(t)$ coincides with $\mathrm{d}_V \psi (u'(t))$ for a.a. $t \in (0,T)$, and 
 \begin{align} \label{eqidenxi}
 \lim_{\varepsilon \rightarrow 0} \int_{0}^{T}  \int_{0}^{t} \langle \xi_{\varepsilon}(s), u'_{\varepsilon}(s)- u_1 \rangle \ds \dt  = \int_{0}^{T} \int_{0}^{t}  \langle   \xi(s), u'(t)- u_1 \rangle \ds \dt . 
 \end{align} 

\section{The viscous limit} \label{sec_viscouslimit}
In this section we discuss the viscous limit $\rho \rightarrow 0$. 
In particular, we prove Theorem \ref{viscous_thm}.i and Theorem \ref{viscous_thm}.ii  in Subsection \ref{subsec:viscgammconv} and Subsection \ref{subsec:viscconv}, respectively. 
Furthermore, we observe that the statement of Theorem \ref{viscous_thm}.iii can be established by arguing as done in Section \ref{sec_causallimit}, but obtaining $\zeta \equiv 0$ as also $\rho \rightarrow 0$.

\subsection{Gamma-convergence of $I_{\rho\varepsilon}$ in the viscous limit} \label{subsec:viscgammconv}
We consider the WIDE functionals $I_{\rho\varepsilon}$ as defined in \eqref{WIDEfunc}, namely
$I_{\rho\varepsilon}: \mathcal{V} \rightarrow(-\infty,\infty]$ such that
\begin{align*}
	I_{\rho\varepsilon}(u)  &  =\left\{
	\begin{array}
		[c]{cl}%
		\displaystyle{\int_{0}^{T}\mathrm{e}^{-t/\varepsilon}\Big(\frac{\varepsilon^2	\rho}{2} \int_{\Omega}|u''(t)|^2 \d x  +\varepsilon
		\nu 	\psi(u'(t))+\phi(u(t))\Big)\mathrm{d}t} & \text{if }u\in K(u_{0}, u_1%
		)\text{,}\\
		\infty & \text{else,}%
	\end{array}
	\right.
\end{align*}
where
\begin{align*}
	K(u_{0}, u_1)  
	=\{u\in H^{2}( 0, T;L^2(\Omega))   \cap W^{1,p}( 0, T;  V ) \cap L^{2}( 0, T;X)   :  u(0)=u_{0}, \rho u'(0)= \rho u_1 
	\} ,
\end{align*}
hence $D(	I_{\rho\varepsilon})= 	K(u_{0}, u_1)  $.
The aim of this subsection is to prove that, by letting $\rho\rightarrow 0$, the sequence of functionals $I_{\rho\varepsilon}$ to $\Gamma$- converges with respect to the strong topology in $\mathcal{V}$ to
\begin{align*}
	\bar{I}_{\varepsilon}(u)  &  =\left\{
	\begin{array}
		[c]{cl}%
		\displaystyle{\int_{0}^{T}\mathrm{e}^{-t/\varepsilon}\Big(\varepsilon
			\nu \psi(u'(t))+\phi(u(t))\Big)\mathrm{d}t} & \text{if }u\in \bar{K}(u_{0}
		)\text{,}\\
		\infty & \text{else,}%
	\end{array}
	\right.
\end{align*}
where
\begin{align*}
	\bar{K}(u_{0})  
	&  =\{W^{1,p}( 0, T;  V ) \cap L^{2}( 0, T;X)   :  u(0)=u_{0}
	\} ,
\end{align*}
hence $D(	\bar{I}_{\varepsilon})= 	\bar{K}(u_{0})  $.


\begin{proof}[Proof of Theorem \ref{viscous_thm}.i]
	\par{$\Gamma-\liminf$ inequality:} 
Consider a sequence $u_\rho$ in $\mathcal{V}$ converging to $u$ in $\mathcal{V}$. 
With no loss of generality we can assume that $\sup_\rho I_{\rho\varepsilon} (u_\rho)< \infty$.
Then, $u_\rho$ is uniformly bounded in $W^{1,p}( 0, T;  V ) \cap L^{2}( 0, T;X)$ so that, by extracting some not relabeled subsequence, one has that $u_\rho$ converges to $u$ weakly in $W^{1,p}( 0, T;  V ) \cap L^{2}( 0, T;X)$.
As a consequence, $\liminf_{\rho \rightarrow 0} I_{\rho\varepsilon}(u_\rho) \geq \liminf_{k \rightarrow \infty} \bar{I}_\varepsilon(u_{\rho_k}) \geq \bar{I}_\varepsilon(u)$.
\par{Existence of a recovery sequence:} Let $u \in \mathcal{V}$. If $u \notin \bar{K}(u_{0})  $ or if $u \in K(u_{0}, u_1)  $, we can choose $u_\rho=u$ and trivially conclude that $  I_{\rho\varepsilon}(u_\rho) \rightarrow \bar{I}_\varepsilon(u)$ as $\rho \rightarrow 0$.
If $u \in \bar{K}(u_{0})  \setminus  K(u_{0}, u_1)$,
we consider a sequence of mollifiers $g_{\tilde{\rho}}$, namely $g_{\tilde{\rho}}(t)= {\tilde{\rho}}^{-1} g(t/{\tilde{\rho}})$ with $g \in C^\infty_c(\mathbb{R})$ and $\int_{\mathbb{R}}g(t)\dt =1$. 
 We define 
$(g_{\tilde{\rho}} \ast u)(t):= \int_{-1}^{T} g_{\tilde{\rho}}(t-s)u(s) \d s $ by setting $u(t)=u_0$ for $t\in (-1,0)$ and $u(t)=0$ for $t \in \mathbb{R}\setminus (-1,T] $ , hence $g_{\tilde{\rho}} \ast u(t)$ is well-defined for every $t \in \mathbb{R}$.
Then we have that $g_{\tilde{\rho}} \ast u \rightarrow u $ converges to $u$ in $W^{1,p}(0,T;V)\cap L^2(0,T;X)$ and in particular $(g_{\tilde{\rho}} \ast u)(0) \rightarrow u_0 $ in $V$ as ${\tilde{\rho}} \rightarrow 0$. 
 
As a next step, we define $u_{\tilde{\rho}}= g_{\tilde{\rho}} \ast u  + u_0 - (g_{\tilde{\rho}} \ast u)(0) +(u_1-(g_{\tilde{\rho}} \ast u)'(0))\zeta^{\tilde{\rho}}$, where $\zeta^{\tilde{\rho}}(t)=t \exp(-t/{\tilde{\rho}})$. 
Note that $(u_{\tilde{\rho}})'(0)=u_1$ and $\zeta^{\tilde{\rho}}
\rightarrow 0 $ in $W^{1,p}(0,T)$ as ${\tilde{\rho}} \rightarrow
0$. As a consequence, $u_{\tilde{\rho}} \in K(u_0,u_1)$ and $u_{\tilde{\rho}} \rightarrow u$ in $W^{1,p}(0,T;V)\cap L^2(0,T;X)$ as ${\tilde{\rho}} \rightarrow 0$.
Furthermore, we have $|(\zeta_{\tilde{\rho}})'' |^2_{L^2(0,T)}\leq C
/\tilde{\rho}^3$ and $(g_{\tilde{\rho}} \ast u)''= g'_{\tilde{\rho}}
\ast u'$, whence we deduce that  $\|(g_{\tilde{\rho}} \ast u)''\|_{L^2(0,T;L^2(\Omega))} \leq \|g_{\tilde{\rho}}'\|_{L^1(0,T)} \| u'\|_{L^2(0,T;L^2(\Omega))} \leq C /\tilde{\rho}$.
By choosing $\tilde{\rho}= \rho^{1/s}$ for some $s>3$, we obtain 
$|I_{\rho\varepsilon}({u}_{\tilde\rho})	-\bar{I}_{\varepsilon}(u_{\tilde\rho})| \leq C
\rho^{1-3/s}$, whose right-hand side goes to zero as $\rho \rightarrow
0$.  On the other hand, having the strong convergences above, by means
of the Dominated Convergence Theorem, we can conclude that
$I_{\rho\varepsilon}({u}_{ \tilde\rho}) \rightarrow \bar{I}_{\varepsilon}({u}) $ as $\rho \rightarrow 0$. 
\end{proof}

\subsection{Viscous limit of the doubly nonlinear wave equation} \label{subsec:viscconv}
Consider any solution  $(u_\rho, \xi_\rho, \eta_\rho)$  to our target problem \eqref{prob1}-\eqref{prob5} belonging to
\[
[W^{1,p}(0,T;V)\cap L^{2}( 0,T ;X)] \times L^{p'}(0,T;V^{\ast})\times   L^{2}\left( 0,T;X^{\ast}\right) =: \mathcal{Y}\text{.}%
\] 
The aim of this subsection is to show that $(u_\rho, \xi_\rho, \eta_\rho)$  converges (with respect to the weak topology of $\mathcal{Y}$ and up to a subsequence) to $(\bar{u}, \bar\xi, \bar\eta)$  satisfying
\begin{align} \label{eq:limitrho}
	\nu \bar\xi + \bar\eta =0 &\quad \text{ in } X^*  \text{ a.e. in } (0,T),\\
	\bar\xi = \mathrm{d}_V\psi(\bar{u}') &\quad \text{ in } V^*  \text{ a.e. in } (0,T),\, \\
    \bar\eta = - \Delta \bar{u} + f(\bar{u}) &\quad \text{ in } X^*  \text{ a.e. in } (0,T),\,
\end{align}
with initial datum $\bar{u}(0)=u_0$.

\begin{proof}[Proof of Theorem \ref{viscous_thm}.ii]
First, we will show that the following estimates hold for $(u_\rho, \xi_\rho, \eta_\rho) \in \mathcal{Y}$ solving \eqref{prob1}-\eqref{prob5}
\begin{align} \label{eq:viscousest1}
	\rho^{1/2} \| u_\rho'\|_{L^\infty(0,T;L^2(\Omega))}+ \| u_\rho'\|_{L^p(0,T;V)} + \| u_\rho\|_{L^\infty(0,T;X)} + \| u_\rho\|_{L^r(0,T;L^r(\Omega))} \leq C, 	
	\\ 	\| \eta_\rho\|_{L^2(0,T;X^*)} + \| \xi_\rho\|_{L^{p'}(0,T;V^*)} \leq C.\label{eq:viscousest2}
\end{align}
By comparison in \eqref{prob1}, we also obtain
\begin{align} \label{eq:viscousest3}
	\rho \|u_\rho''\|_{L^{2}(0,T;X^{\ast})+L^{p'}(0,T;V^{\ast})} \leq C \text{.}
\end{align}
In order to prove \eqref{eq:viscousest1}-\eqref{eq:viscousest2}, we proceed by arguing on increments.
For any arbitrary constant $\tau > 0$, we define a backward difference operator $\delta_\tau$ by
\[
\delta_\tau\chi (t) = \frac{\chi(t) - \chi(t-  \tau)}{\tau}
\]
for functions $\chi$ defined on $[0,T]$ with values in a vector space and for $t \geq \tau$. 
Test \eqref{prob1} with $\delta_\tau u_\rho$, integrate on $(\tau,t)$ and by parts, obtaining
\begin{align*}
	&-\frac{\rho}{2\tau} \int_{t -\tau}^t \|u_\rho'(s)\|^2 \ds + \frac{\rho}{2\tau} \int_{0}^\tau \|u_\rho'(s)\|^2  \ds
	+  \langle u_\rho'(t), \frac{\rho}{\tau}(u_\rho(t)- u_\rho(t-\tau))\rangle \\
	& - \langle u_\rho'(\tau)  , \frac{\rho}{\tau} (u_\rho(\tau)-
   u_\rho(0)) \rangle  - \frac{\rho}{2\tau}  \int_\tau^{t} \|u_\rho'(s- \tau) - u_\rho'(s)\|^2 \ds+ \int_{\tau}^{t} \langle\nu \xi_\rho(s) , \delta_\tau u_\rho(s)\rangle \ds \\
	& + \frac1{2\tau} \int_{t -\tau}^t \|\nabla u_\rho(s)\|^2 \d s - \frac1{2\tau} \int_{0}^\tau \|\nabla u_\rho(s)\|^2 \ds
	-  \frac1{2\tau} \int_\tau^{t}\|\nabla u_\rho(s- \tau) - \nabla u_\rho(s)\|^2  \ds\\
	& + \int_{\tau}^{t}\langle f(u_\rho(s)) , \delta_\tau u_\rho(s)\rangle \ds =0.
\end{align*}
By letting $\tau \rightarrow 0$, the regularity of $u_\rho$ and the  Lebesgue  Differentiation Theorem allow us to conclude that
\begin{align*}
	&-\frac{\rho}{2}  \|u_\rho'(t)\|^2  + \frac{\rho}{2}  \|u_\rho'(0)\|^2
	+ \rho \|u_\rho'(t) \|^2
	- \rho \|u_\rho'(0)\|^2  + 0+ \int_{0}^{t} \langle \nu \xi_\rho(s) ,  u_\rho'(s)\rangle \ds \\
	& + \frac1{2}  \|\nabla u_\rho(t)\|^2 - \frac1{2} \|\nabla u_\rho(0)\|^2
	+ 0 
	+ \int_{0}^{t}\langle f(u_\rho(s)) ,u_\rho'(s)\rangle \ds =0
\end{align*}
for almost every $t \in (0,T)$. This can be rewritten as 
\begin{align*}
	&\frac{\rho}{2}  \|u_\rho'(t)\|^2  + \int_{0}^{t} \langle\nu \xi_\rho(s) ,  u_\rho'(s)\rangle \ds + \frac1{2}  \|\nabla u_\rho(t)\|^2
	+ \int_{0}^{t}\langle f(u_\rho(s)) ,u_\rho'(s)\rangle \ds \\&=
  \frac{\rho}{2}  \|u_1\|^2  + \frac1{2} \|\nabla
  u_0\|^2  ,
\end{align*}
whence we obtain \eqref{eq:viscousest1}-\eqref{eq:viscousest2} by using the assumptions of Section \ref{sec_assumptions}.


From \eqref{eq:viscousest1}-\eqref{eq:viscousest3}, up to not
relabeled subsequences, we deduce the following convergences:
\begin{align}
	u_{\rho}  &  \rightarrow \bar{u} \text{ weakly in }W^{1,p}(0,T;V)\text{ and
		strongly in } C([0,T]; V),\label{convurho2} \\
	u_{\rho}  &  \rightarrow \bar{u} \text{ weakly star in }L^{\infty}(0,T;X) \text{,} \label{convuxrho2}\\
	\xi_{\rho}  &  \rightarrow \bar\xi\text{ weakly in }L^{p'%
	}(0,T;V^{\ast})\text{,}\label{c1rho2}\\
	\rho u_\rho''&  \rightarrow \bar\zeta \text{ weakly in }%
	L^{2}(0,T;X^{\ast})+L^{p'}(0,T;V^{\ast}) \text{,}\label{c2rho2}\\
	\eta_\rho  &  \rightarrow \bar\eta \text{ weakly in }L^{2}\left(
	0,T;X^{\ast}\right).\label{c3rho2}
\end{align}
The initial condition $\bar{u}(0)=u_0$ follows directly from \eqref{convurho2}.
The identification of the limit $\bar\zeta \equiv 0$ can be obtained as in Section \ref{sec_causallimit}.
Using the same argument of Section \ref{sec_causallimit}, we can also
deduce that  $\bar\eta=- \Delta \bar{u}  + f(\bar{u})$. 
It just remains to identify the limit $\bar\xi$. 
In order to achieve this, we first consider \eqref{ineqcausal}, namely
\begin{align*} 
	\int_{0}^{T} \int_{0}^{t} \langle \nu \xi_\rho (s), u_\rho'(s)- u_1 \rangle \ds \dt 
	\leq   
	-\int_0^T \phi(u_\rho(t))\dt  + T\phi(u_0) + \int_{0}^T  \int_0^t \langle \eta_\rho(s), u_1 \rangle_X \ds \dt \notag 	
\end{align*}
due to \eqref{eqidenxi}. Then, convergences \eqref{convuxrho2} and \eqref{c3rho2} ensure that
\begin{align*} 
	&\limsup_{\rho \rightarrow 0}\int_{0}^{T} \int_{0}^{t} \nu \langle \xi_\rho (s), u_\rho'(s)- u_1 \rangle \ds \dt \\
	&\leq   
	-\int_0^T \phi(\bar{u} (t))\dt  + T\phi(u_0) + \int_{0}^T  \int_0^t \langle \bar\eta(s), u_1 \rangle \ds \dt \\
	& \leq  - \int_{0}^T  \int_0^t \langle \bar\eta(s), \bar{u}' (s) - u_1 \rangle_X \ds \dt = \int_{0}^T  \int_0^t \nu \langle
   \bar\xi(s), \bar{u}' (s) - u_1 \rangle_X \ds \dt ,
\end{align*}
where we used \eqref{eq:limitrho}. 
From the demiclosedness of the maximal monotone operator $u
\rightarrow \mathrm{d}_V \psi(u(\cdot))$ in ${\mathcal V} \times
{\mathcal V}^*$, we can conclude by Remark \ref{rem:limsup} and Lemma \ref{lemma:doublelimsup}
 that $\bar\xi(t)$ coincides with $\mathrm{d}_V \psi (\bar{u} '(t))$ for a.a. $t \in (0,T)$.
\end{proof}

\begin{remark}
	Observe that, since we do not assume $f$ to be Lipschitz
        continuous, solutions $u_\rho$ to \eqref{prob1}-\eqref{prob5}
        (as well as solutions $\bar{u}$ to
        \eqref{probvisc1}-\eqref{probvisc4}) may be not unique.
	For this reason, there might exist solutions $u_\rho$ to \eqref{prob1}-\eqref{prob5} which cannot be recovered by means of the WIDE approach, namely, that are not limits as $\varepsilon \rightarrow 0$ of sequences of solutions to the regularized problem \eqref{euler:lambda}-\eqref{euler:4}.
	Note that Theorem \ref{viscous_thm}.ii establishes the convergence (with respect to the weak topology of $\mathcal{Y}$ and up to a subsequence) of any solution $(u_\rho, \xi_\rho, \eta_\rho) \in \mathcal{Y} $ to \eqref{prob1}-\eqref{prob5}  towards a solution $(\bar{u}, \bar{\xi}, \bar{\eta})$ to \eqref{probvisc1}-\eqref{probvisc4} as $\rho \rightarrow 0$. 
	This result implies that the same holds true for any $(u_\rho, \xi_\rho, \eta_\rho) \in \mathcal{Y} $ solution to \eqref{prob1}-\eqref{prob5} obtained as causal limit in Section \ref{sec_causallimit}. This could be proved also starting from the estimates derived in Subsection \ref{subsec_a priori estimates} which are uniform in $\rho>0$. 
\end{remark}

\section*{Acknowledgments}
G.A.~is supported by JSPS KAKENHI Grant Numbers JP21KK0044,
JP21K18581, JP20H01812 and JP20H00117 and also by the Research
Institute for Mathematical Sciences, an International Joint
Usage/Research Center located in Kyoto University.
A.M. is funded by the Deutsche Forschungsgemeinschaft (DFG, German Research Foundation) under Germany’s Excellence Strategy - GZ 2047/1, Projekt-ID 390685813.
U.S. is partially funded by the  Austrian
Science Fund grants 10.55776/I5149, 10.55776/F65, 10.55776/I4354, and 10.55776/P32788. 

\appendix

\section{Proofs of \eqref{aaa}}\label{apdx}

Integrating both sides of \eqref{euler:lambda} over $(t,T)$ and using \eqref{euler:3} and \eqref{euler:4}, we deduce that
\begin{align*}
\rho \varepsilon^2 u'''_{\varepsilon \lambda \mu}(t)
&= 2 \rho \varepsilon u''_{\varepsilon \lambda \mu}(t)
 + \rho u'_{\varepsilon \lambda \mu}(T) - \rho u'_{\varepsilon \lambda \mu}(t) + \int^T_t \eta_{\varepsilon \lambda \mu}(\tau) \, \d \tau\\
&\quad + \varepsilon \xi_{\varepsilon \lambda \mu}(t) + \int^T_t \xi_{\varepsilon \lambda \mu}(\tau) \, \d \tau\\
&\to 2 \rho \varepsilon u''_{\varepsilon \lambda}(t)
 + \rho u'_{\varepsilon \lambda}(T) - \rho u'_{\varepsilon \lambda}(t) + \int^T_t \eta_{\varepsilon \lambda}(\tau) \, \d \tau\\
&\quad + \varepsilon \xi_{\varepsilon \lambda}(t) + \int^T_t \xi_{\varepsilon \lambda}(\tau) \, \d \tau =: z(t) \quad \mbox{ strongly in } X^*
\end{align*}
as $\mu \to 0_+$ for each $t \in [0,T]$. On the other hand, we recall that
$$
u'''_{\varepsilon \lambda \mu} \to u'''_{\varepsilon \lambda} \quad \mbox{ weakly in } W^{1,2}(0,T;X^*) 
$$
as $\mu \to 0_+$. Therefore we have
$$
z(t) = \rho \varepsilon^2 u'''_{\varepsilon \lambda}(t) \quad \mbox{ for } \ t \in [0,T],
$$
which implies
$$
u'''_{\varepsilon \lambda \mu}(t) \to u'''_{\varepsilon \lambda}(t) \quad \mbox{ weakly in } X^*
$$
for each $t \in [0,T]$, and hence, \eqref{euler:3l} particularly follows.

Repeating the same argument above with $X^*$ replaced by $X^* + L^q(\Omega)$, one can also verify \eqref{bbb}.

\end{document}